\documentclass[a4paper, reqno, 14pt]{amsart}
\usepackage[utf8]{inputenc}
\usepackage{graphicx}

\usepackage[usenames,dvipsnames]{color}

\usepackage{amsthm,amsfonts,amssymb,amsmath}
\usepackage[all]{xy}
\usepackage{xr-hyper}
\usepackage[colorlinks=
   citecolor=Black,
   linkcolor=Red,
   urlcolor=Blue]{hyperref}
\usepackage{verbatim}
\usepackage{pdfsync}

\newcommand{\BA}{{\mathbb {A}}}

\newcommand{\BF}{{\mathbb {F}}}

\newcommand{\BP}{{\mathbb {P}}}
\newcommand{\BQ}{{\mathbb {Q}}}
\newcommand{\BR}{{\mathbb {R}}}

\newcommand{\BX}{{\mathbb {X}}}

\newcommand{\BZ}{{\mathbb {Z}}}

\newcommand{\CB}{{\mathcal {B}}}

\newcommand{\CE}{{\mathcal {E}}}
\newcommand{\CF}{{\mathcal {F}}}

\newcommand{\CL}{{\mathcal {L}}}
\newcommand{\CM}{{\mathcal {M}}}
\newcommand{\CN}{{\mathcal {N}}}
\newcommand{\CO}{{\mathcal {O}}}

\newcommand{\CT}{{\mathcal {T}}}

\newcommand{\CV}{{\mathcal {V}}}

\newcommand{\diag}{{\mathrm{diag}}}

\newcommand{\End}{{\mathrm{End}}}

\newcommand{\Gal}{{\mathrm{Gal}}\,}
\newcommand{\GL}{{\mathrm{GL}}}

\newcommand{\Hom}{{\mathrm{Hom}}\,}

\newcommand{\Lie}{{\mathrm{Lie}}\,}

\newcommand{\SL}{{\mathrm{SL}}}
\newcommand{\Spec}{{\mathrm{Spec}}\,}

\newcommand{\Sp}{{\mathrm{Sp}}}
\newcommand{\SU}{{\mathrm{SU}}}

\newcommand\bk{\boldsymbol{k}}

\newcommand{\lra}{\longrightarrow}

\newcommand{\bs}{\backslash}

\newtheorem{theorem}{Theorem}[section]
\newtheorem{proposition}[theorem]{Proposition}
\newtheorem{lemma}[theorem]{Lemma}

\newtheorem{corollary}[theorem]{Corollary}

\theoremstyle{definition}
\newtheorem{definition}[theorem]{Definition}

\newtheorem{remark}[theorem]{Remark}

\numberwithin{equation}{section}

\begin{document}


\title[Supersingular locus]{The supersingular locus of the Shimura variety for ${\rm GU}(1, n-1)$ over a ramified prime}
\author{Michael Rapoport, Ulrich Terstiege and Sean Wilson}

\date{\today}
\maketitle

\tableofcontents

\section{Introduction}
Consider the reduction modulo a prime ideal of an integral model of a Shimura variety. The present paper is a contribution to  the general problem of giving a {\it concrete description} of its {\it basic locus}.  This question has been addressed in the case of the Siegel moduli spaces by Katsura/Oort,  Li/Oort, Kaiser, Richartz, and Kudla/Rapoport. We refer to the introduction of \cite {V} for the precise references. In this case the basic locus coincides with the {\it supersingular locus}. In the case of Hilbert-Blumenthal moduli spaces, where again the basic locus coincides with the supersingular locus, there are results by Bachmat/Goren, by Goren, by Goren/Oort, by Stamm and by Yu. Again we refer to the introduction of \cite{V} for the precise references. For the Shimura variety for ${\rm GU}(1, n-1)$ at an inert prime, where again  the basic locus coincides with the supersingular locus, there are the results of Vollaard \cite{V} and by Vollaard/Wedhorn \cite{VW}. It is our purpose here to prove structure theorems,   analogous to \cite{VW},  for the Shimura variety for ${\rm GU}(1, n-1)$ in the case of a ramified prime. 

By the uniformization theorem of \cite{RZ}, the general problem can be seen as a special case of the general problem of describing the underlying reduced scheme $\CN_{\rm red}$ of any RZ-space $\CN$. The experience of the work done so far on this problem seems to roughly indicate that the set of irreducible components of $\CN_{\rm red}$ should be describable in terms of the Bruhat-Tits simplicial complex  associated to the corresponding algebraic group $J$ over $\BQ_p$. Furthermore, each irreducible component should be related to some Deligne-Lusztig variety, although it only rarely will be actually isomorphic to a Deligne-Lusztig variety. It should be pointed out that we can solve this kind of problem in only a limited number of cases. The analogue of this problem in the equi-characteristic case is a question about the structure of affine Deligne-Lusztig varieties (ADLV) \cite{GH}. U.~G\"ortz and X.~He have informed us\footnote{oral communication} that they produced an essentially  complete list of ({\it basic}) ADLV for which the set of irreducible components are described by the Bruhat-Tits complex of $J$, and in which each irreducible component is isomorphic to a Deligne-Lusztig variety. It would be interesting to transpose their list to the unequal characteristic case. The case of the RZ-space associated to an unramified hermitian space of signature  $(2, 2)$ has been solved by B.~Howard and G.~Pappas \cite{HP}. 

The Shimura variety considered in this paper is associated to ${\rm GU}(1, n-1)$ over a ramified prime $p$, for the parahoric level structure at $p$ given by a selfdual lattice in the localization at $p$ of the corresponding hermitian vector space. We note that only when $n$ is odd, is the corresponding parahoric subgroup {\it special} in the sense of Bruhat-Tits. This Shimura variety has bad reduction at $p$, for all $n$. There is a resolution of singularities through the {\it Kr\"amer model}, \cite{Kr}, but our results only concern the original moduli problem. We should point out that there is always, for ${\rm GU}(1, n-1)$ over a ramified prime $p$, for even or odd $n$, a special parahoric subgroup where the corresponding Shimura variety has good reduction ({\it exotic good reduction}, cf. \cite{PR3, A, Ri}). It would be very interesting to have  results on the supersingular locus analogous to the results in this paper in these cases as well. 

For the Shimura varieties attached to ${\rm GU}(1, n-1)$, the problem of describing the basic locus is of relevance to the study of special subvarieties and their intersections on them, and in fact of two sorts: to the study of special divisors on them as in \cite{KR}; and also to the study of special subvarieties on the product of two such Shimura varieties, for $n$ and for $n-1$, relevant to the Arithmetic Gan-Gross-Prasad Conjecture and the Arithmetic Fundamental Lemma conjecture of W. Zhang \cite{Zh}. We refer to the papers \cite{T,T2, RTZ} for examples of such use in the inert case. The results of \cite{VW} are also used in the work of Xu Shen \cite{XS} on the fundamental domain for the natural group action on the RZ-space corresponding to ${\rm GU}(1, n-1)$ over an unramified prime, in analogy to Fargues' work on the RZ-space for $\GL_n$, \cite{Fa}. 

We now state our main results. We start with the local situation. Let $E$ be a ramified quadratic extension of $\BQ_p$. We denote by $\BF$ an 
algebraic closure of $\BF_p$, and by $W=W(\BF)$ its ring of Witt vectors and by $W_\BQ$ its fraction
field. Let $\breve{E} = W_\BQ \otimes_{\BQ_p} E$ and let $\CO_{\breve{E}} = W \otimes_{\BZ_p} \CO_E$ be its ring
of integers. Let $(\BX, \iota_\BX)$ be a fixed supersingular $p$-divisible group of dimension $n$ and height $2n$ over $\BF$
with an action $\iota_\BX : \CO_E \lra \rm End(\BX)$. Let $\lambda_\BX$ be a principal quasi-polarization such that its
Rosati involution induces on $\CO_E$ the non-trivial automorphism over $\BQ_p$.  We refer to section \ref{section.Themodulispace} for the definition of the formal scheme $\CN$ over ${\rm Spf}\, \CO_{\breve E}$ which represents the formal moduli problem of certain triples $(X, \iota, \lambda)$ together with a quasi-isogeny  of its special fiber with $(\BX, \iota_\BX, \lambda_\BX)$. We denote   by $\CN^0$ its open and closed locus where the quasi-isogeny has height $0$, and by $\bar\CN^0$ the  fiber of $\CN^0$ over $\BF$. Our {\it local} main theorem describes the underlying reduced scheme $\bar\CN^0_{\rm red}$ of $\bar\CN^0$. 

To the triple $(\BX, \iota_\BX, \lambda_\BX)$, there is associated a hermitian vector space $C$ of dimension  $n$ over $E$. A $\CO_E$-lattice $\Lambda$ in $C$ is called a vertex lattice of type $t=t(\Lambda)$, if the dual lattice $\Lambda^\sharp$ is contained in $\Lambda$ and the factor module $\Lambda/\Lambda^\sharp$ is a $\BF_p$-vector space of dimension $t$.  It turns out that the type is an even integer, and that all even integers $0\leq t\leq n$ occur as types of vertex lattices, except when $n$ is even and the hermitian space $C$ is non-split, in which case all even $t$ with $0\leq t\leq n-2$ occur but $t=n$ does not occur.

\begin{theorem}\label{mainthIntro}
  i) There is a stratification of $\bar\CN^{0}_{\rm{red}}$ by locally closed subschemes 
given by $$
\bar\CN^{0}_{\rm{red}}= \biguplus_{\Lambda}\CN^{\circ}_{\Lambda}\, , 
$$
where $\Lambda$ runs over all vertex lattices. Each stratum $\CN^{\circ}_{\Lambda}$ is isomorphic to a Deligne-Lusztig variety attached to the symplectic group of size $t(\Lambda)$ over $\BF_p$ and a standard Coxeter element. The closure of any $\CN^{\circ}_{\Lambda}$ in $\bar\CN^{0}_{\rm{red}}$ is given by 
$$
\overline{\CN^{\circ}_{\Lambda}}=\biguplus_{ \Lambda'\subseteq \Lambda}\CN^{\circ}_{\Lambda'}=\CN_{\Lambda}\, ,
$$
and this is a projective variety of dimension   $t(\Lambda)/2$, which is normal 
with isolated singularities when $t(\Lambda)\geq 4$. When $t(\Lambda)=2$, then $\overline{\CN^{\circ}_{\Lambda}}\cong \BP^1$. 

\smallskip

\noindent We call this stratification the \emph{Bruhat-Tits stratification} of $\bar\CN^{0}_{\rm{red}}$. 

\smallskip

\noindent ii) The incidence complex of the Bruhat-Tits stratification can be identified with the simplicial complex $\CT$ of vertex lattices.  
In particular,  $\bar\CN^{0}_{\rm{red}}$ is connected. When $n$ is odd, or $n$ is even and $C$ is non-split, then this can also be identified with the Bruhat-Tits simplicial complex of $\SU(C)(\BQ_p)$.

\smallskip

\noindent iii) The $\CN_{\Lambda}$, for $\Lambda$ of  maximal type,  are the irreducible components of $\bar\CN^{0}_{\rm{red}}$. The maximal value $t_{\max}$ of $t(\Lambda)$ is given by
$$t_{\max}=
\begin{cases}
   n  & \text{if $n$ is even and $C$ is split,}\\  
   n-2 & \text{if $n$ is even and  $C$ is non-split,}\\    
      n-1 & \text{if $n$ is odd.}\\
\end{cases}$$
 
 \noindent Furthermore,  $\bar\CN^{0}_{\rm{red}}$ is of pure dimension $t_{\max}/2$. 
  
\smallskip

\noindent iv) Let $\Lambda$ be a vertex lattice of type $2m$ and let $V=\Lambda/\Lambda^{\sharp}$. The scheme  $\overline{\CN^{\circ}_{\Lambda}}=\CN_{\Lambda}$ is isomorphic to the projective closure of a generalized Deligne-Lusztig variety $S_V$. Furthermore the stratification by Deligne-Lusztig varieties $S_V= \biguplus_{i = 0}^m S_i$ of Proposition \ref{decomp} is compatible with the stratification given in i) in the following sense: Under the isomorphism  $\overline{\CN^{\circ}_{\Lambda}}\cong S_V$ (given by Proposition \ref{iso}), for any $i\leq m$ and any inclusion of index $m-i$ of a vertex lattice $\Lambda'\subseteq \Lambda$, the subscheme $\CN^{\circ}_{\Lambda'}$ is identified with an irreducible component of $S_i$, and all  irreducible components of $S_i$ arise in this way. In particular, $S_m$  is identified with $\CN^{\circ}_{\Lambda}$.
\end{theorem}
The results in this ramified case are therefore quite analogous to the unramified case. What is remarkable is that in the ramified case the strata are Deligne-Lusztig varieties attached to {\it split} groups (namely, symplectic groups), whereas in the unramified case, they are Deligne-Lusztig varieties attached to {\it non-split} groups (namely, unitary groups). On the other hand, in the ramified case the closures of strata have singularities, whereas they are smooth in the unramified case.  

Let us discuss the low-dimensional cases. When $n=2$, there are two possibilities, depending on whether $C$ is split or non-split. When $C$ is non-split, the scheme $\bar\CN^0_{\rm red}$ consists of a single point. When $C$ is split, then $\bar\CN^0_{\rm red}$ is one-dimensional, and is a union of projective lines.  The {\it dual graph} of  this curve is given by  the Bruhat-Tits complex of $\SU(C)(\BQ_p)$, which  coincides with the Bruhat-Tits tree of $\SL_2(\BQ_p)$.  In fact, it is proved in \cite{KR3} that, in this case, the formal scheme $\CN^0$ is isomorphic to the Drinfeld formal scheme $\widehat\Omega^2_{\BZ_p}\times_{{\rm Spf}\, \BZ_p}{\rm Spf}\,  \CO_{\breve E}$, and the reduced locus coincides in this case with its special fiber, with its well-known intersection pattern. When $n=3$, then $\bar\CN^0_{\rm red}$ is one-dimensional. All its irreducible components are projective lines, and the Bruhat-Tits complex of $\SU(C)(\BQ_p)$ describes the incidence graph of the stratification. The latter two cases are described by {\it figure} 1. 

Now we state our global results. Let $\bk$ be an imaginary quadratic field, and let $p>2$ be a prime number that ramifies in $\bk$. Our results concern Shimura varieties $\underline{\rm Sh}^V_K$ associated to a hermitian $\bk$-vector space $V$ of signature $(1,n-1)$, and open compact subgroups $K$ of the associated group of unitary similitudes $G^V(\BA_f)$ which are of the form $K=K^pK_p$, where $K_p$ is the stabilizer of a self-dual lattice in $V\otimes_{\bk}{\bk_p}$. These Shimura varieties have integral models $\CM_{K^p}$ over the localization $\CO_{\bk_{(p)}}$ of $\CO_{\bk}$ at $p$. These models are defined by formulating a moduli problem of principally polarized abelian varieties with action of $\CO_{\bk}$ and a level-$K^p$-structure, cf. section \ref{section.Shimuravarieties}.  We denote by $\CM_{K^p}^{ss}$ the supersingular locus of $\CM_{K^p}\times_{\Spec \CO_{\bk_{(p)}}}\Spec\, \overline{\BF}_p$. We use the notation introduced in the formulation of Theorem \ref{mainthIntro}, relative to the quadratic extension $E=\bk_p$ of $\BQ_p$ (then the hermitian space $C$ has invariant that differs from that of $V_p$ by the factor $(-1)^{n-1}$). The algebraic group $I^V$ is an inner form of the unitary group of $V$. For all other notation we refer to the body of the text.

\begin{theorem}\label{mainthglobalIntro}
\noindent (i) The supersingular locus $\CM^{ss}_{K^p}$ is of pure dimension $t_{\rm max}/2$. It is stratified by Deligne-Lusztig varieties associated to symplectic groups over $\BF_p$ of size varying between $0$ and $t_{\rm max}$ and standard Coxeter elements. The incidence complex of the stratification can be identified with the complex $\coprod_{j=1}^m \Gamma_j\backslash \CT$. 

\smallskip

\noindent (ii) There is a natural bijection
$$
\{\text{irreducible components of } \CM^{ss}_{K^p}\} \leftrightarrow I^V(\BQ)\setminus \big(J(\BQ_p)^o/K_{J,p}\times { G^V}(\BA_f^p)^o/K^p\big).
$$

\smallskip

\noindent  (iii) There is a natural bijection $$
\{\text{connected components of }  \CM^{ss}_{K^p}\} \leftrightarrow I^V(\BQ)\setminus  { G^V}(\BA_f^p)^o/K^p.
$$

\end{theorem}

\begin{figure}
\includegraphics[width=3.3cm]{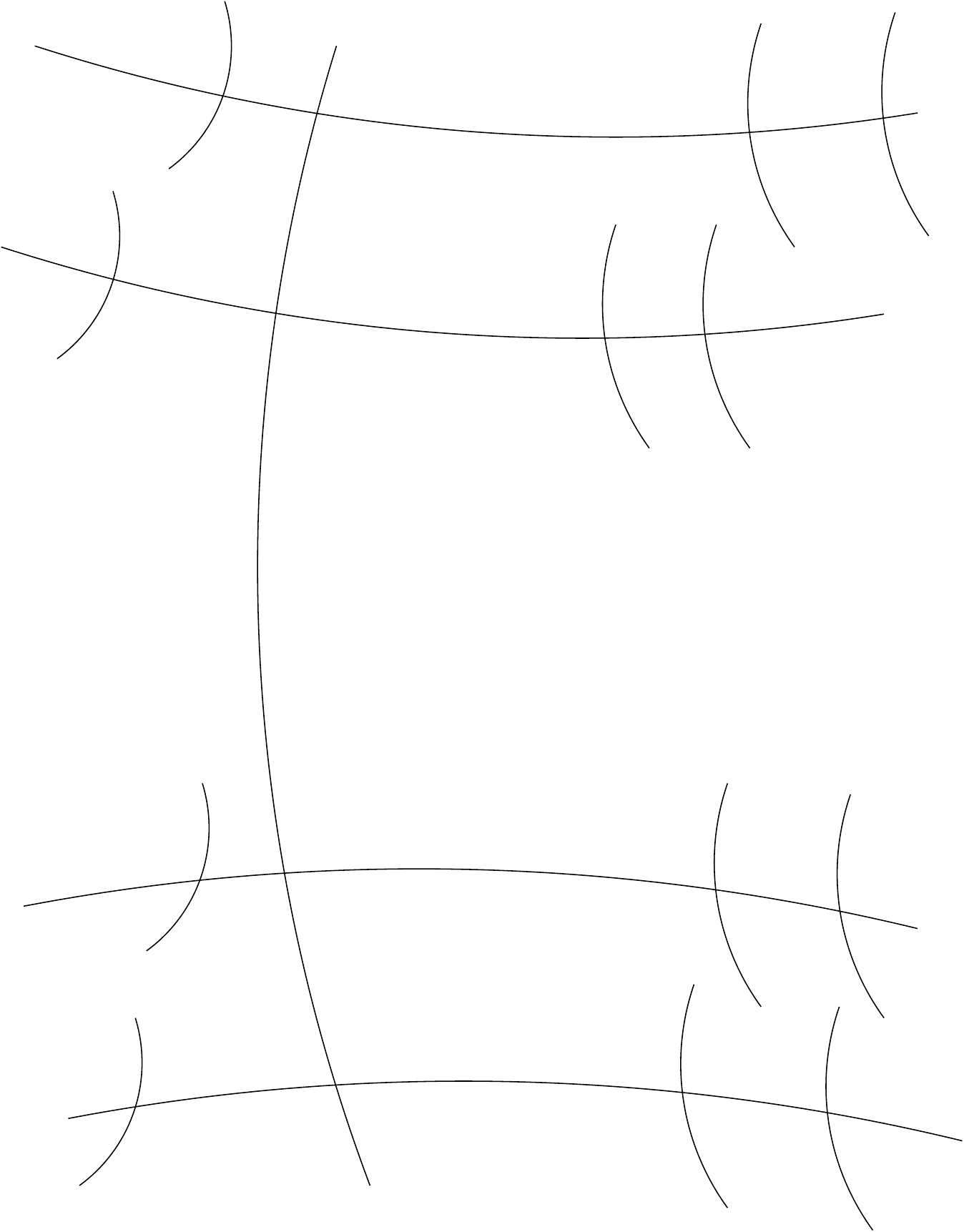}\quad\quad\quad\quad\quad\quad\quad\quad\includegraphics[width=2cm]{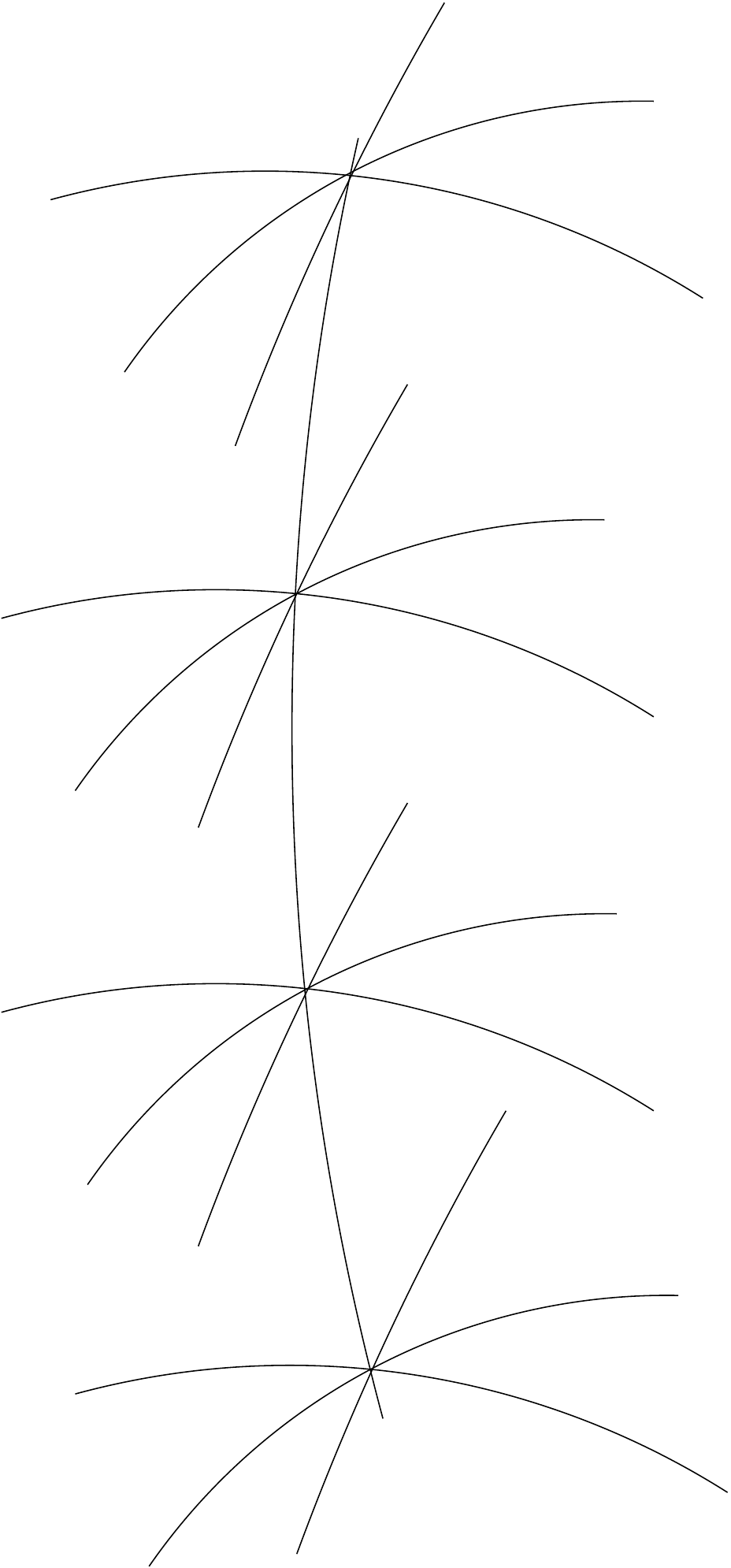}

\caption{\it On the left is the case $n=2$, $C$ split; on the right is the case $n=3$.  All irreducible components are projective lines. On the left, each projective line contains $p+1$ special points, and each special point lies on two projective lines. On the right, each projective line contains $p+1$ special points,  and each special point lies on $p+1$ projective lines. }
\end{figure}

We thank N. Perrin and T. Richarz for pointing out the paper \cite{BP} by Brion and Polo to us, and related discussions. We also thank 
U.~G\"ortz, B.~Howard, S.~Kudla and G.~Pappas for pointing out  mistakes in a first version of this paper, and  for their help in correcting them. This work was supported by the SFB/TR 45 `Periods, Moduli Spaces and Arithmetic of Algebraic Varieties' of the DFG (German Research Foundation).
\medskip

{\bf Notation} Throughout the paper, $p$ denotes an odd prime number. Furthermore, $E$ is a ramified quadratic extension of $\BQ_p$; the action of the non-trivial automorphism in $\Gal(E/\BQ_p)$ is denoted by $a\mapsto \bar a$.

\section{The moduli space}\label{section.Themodulispace} 
We denote by $E$ a ramified quadratic extension of $\BQ_p$. We fix a uniformizer $\pi$ of $E$ such that 
$\pi_0 = \pi^2 \in \BQ_p$ is a uniformizer. Since $p\neq 2$, this is always possible. We denote by $\BF$ an 
algebraic closure of $\BF_p$, and by $W=W(\BF)$ its ring of Witt vectors and by $W_\BQ$ its fraction
field. We denote by $\sigma$ the Frobenius automorphism on $\BF$, on $W$, and on $W_\BQ$.

Let $\breve{E} = W_\BQ \otimes_{\BQ_p} E$ and let $\CO_{\breve{E}} = W \otimes_{\BZ_p} \CO_E$ be its ring
of integers. Let $\sigma = \sigma \otimes \rm id$ on $\breve{E}$. We denote by $\psi_0 : E \lra \breve{E}$
the natural embedding, and by $\psi_1 = \psi_0 \, \circ \, ^{-}$ its conjugate.

Let $\rm Nilp$ be the category of $\CO_{\breve{E}}$-schemes $S$ such that $\pi \cdot \CO_S$ is a locally
nilpotent ideal sheaf. For $S \in \rm Nilp$, we denote by $\bar S = S \times_{\rm Spec\, \CO_{\breve{E}}} 
\rm Spec \, \BF$ its reduction modulo $\pi$.

Let $(\BX, \iota_\BX)$ be a fixed supersingular $p$-divisible group of dimension $n$ and height $2n$ over $\BF$
with an action $\iota_\BX: \CO_E \lra \rm End(\BX)$. Let $\lambda_\BX$ be a principal quasi-polarization such that its
Rosati involution induces on $\CO_E$ the non-trivial automorphism over $\BQ_p$.  

Let $N$ be the rational Dieudonn\'{e} module of $\BX$. Then $N$ has an action of $E$ and a skew-symmetric 
$W_\BQ$-bilinear form $\langle\, ,\, \rangle,$ satisfying
\begin{equation*}
\begin{aligned}
 \langle F x,y \rangle &= \langle x, Vy \rangle^{\sigma} \, ,\quad\\
 \langle \iota (a) x,y \rangle &= \langle x, \iota (\bar a) y \rangle \, , \qquad a \in E.
 \end{aligned}
\end{equation*}
We denote by $\Pi$ the endomorphism induced by $\pi$ on $N$.

Fix $n \geq 2$. Let $\CN$ be the set-valued functor on $\rm Nilp$ which associates to $S\in \rm Nilp$ the set of 
isomorphism classes of quadruples $(X, \iota, \lambda, \varrho)$. Here $X$ is a $p$-divisible group over $S$,
and $\iota : \CO_E \lra \End (X)$ is a homomorphism satisfying the following two conditions
\begin{equation}
\begin{aligned}
 {\rm char} (\iota (a) \vert {\rm Lie} \, X) &= (T - \psi_0(a)) \cdot (T - \psi_1 (a))^{n-1} \, .\\
\bigwedge^n (\iota(\pi)-\pi \vert {\rm Lie} \, X) &= 0\, , \quad \bigwedge^2 (\iota(\pi)+\pi \vert {\rm Lie} \, X) = 0 \, , \text{ if $n \geq 3$.}
\end{aligned}
\end{equation}
The condition on the characteristic polynomial is the {\it Kottwitz condition}, the condition on the wedge products is the {\it Pappas condition}.
Furthermore, $\lambda : X \lra {X}^\vee$ is a principal quasi-polarization whose associated Rosati involution
induces on $\CO_E$ the non-trivial automorphism over $\BQ_p$. Finally, $\varrho : X \times_S \bar S \lra
\BX \times_{\rm Spec \, \BF} \bar S$ is a $\CO_E$-linear quasi-isogeny such that 
$\lambda$ and $\varrho^* (\lambda_\BX)$ differ locally on $\bar S$ by a factor in $\BQ_p^{\times}$. An 
isomorphism between two quadruples $(X, \iota, \lambda, \varrho)$ and $(X', \iota', \lambda', \varrho')$ is
given by an $\CO_E$-linear isomorphism $\alpha : X \lra X'$ such that 
$\varrho' \circ (\alpha \times_S \bar S) = \varrho$ and such that $\alpha^*(\lambda')$ is a 
$\BZ^\times_p$-multiple of $\lambda$.
\begin{proposition}
The functor $\CN$ is representable by a separated formal scheme $\CN$, locally formally of finite type over 
${\rm Spf}\, \CO_{\breve{E}}$. Furthermore, $\CN$ is flat over $\CO_{\breve{E}}$. It is
formally smooth over $\CO_{\breve{E}}$ in all points of the special fiber except those corresponding to points 
$(X, \iota, \lambda, \varrho) \in \CN (\BF)$, where $ {\rm Lie}(\iota (\pi))= 0$ {\rm (these form 
an isolated set of points).} 
\end{proposition}
\begin{proof}
The representability follows from \cite{RZ}. The assertions concerning flatness and formal smoothness follow from  \cite{P}, 4.5.  
\end{proof}
We denote by $\CN^0$ the open and closed formal subscheme of $\CN$ where the height of $\rho$ is zero. Let $\bar \CN^0$ be the reduction modulo $\pi$ of $\CN^0$. Then the moduli functor  $\bar \CN^0$  has a simpler description. Namely, $\bar \CN^0$ parametrizes isomorphism
classes of quadruples $(X, \iota, \lambda, \varrho)$, where $X$ is a $p$-divisible group of height $2n$ and 
dimension $n$, and where $\iota : \CO_E \lra \End \, (X)$ is an action of $\CO_E$ on $X$, such that for
$n \geq 3$
\begin{equation}\label{wedcond}
 \bigwedge^2 (\iota(\pi) \vert {\rm Lie} \, X) = 0 \, ,
\end{equation}
and where $\lambda$ is a principal quasi-polarization whose Rosati involution induces on $\CO_E$ the non-trivial
automorphism over $\BQ_p$, and where $\varrho : X \lra \BX \times_{\rm Spec \, \BF} S$ is a quasi-isogeny of
height $0$ which is $\CO_E$-linear and such that $\varrho^*(\lambda_\BX)$ and $\lambda$ differ locally on
$S$ by a factor in $\BZ^\times_p$ (we note that $\psi_0$ and $\psi_1$ are identical modulo $\pi$).
\begin{proposition}\label{pointdescr}
 Recall the rational Dieudonn\'{e} module $N$, with its action by $E$ and its skew-symmetric form 
$\langle\, ,\, \rangle$. Associating to a point in $\bar \CN^0 (\BF)$ its Dieudonn\'e module defines a bijection of $\bar \CN^0 (\BF)$  with the set of $W$-lattices 
\begin{equation*}
 \{ M \subset N \mid M^\vee = M, \Pi \, M \subset M, \, pM \subset V M \subset^n M, VM \subset^{\leq 1} VM + \Pi \, M \} \, .
\end{equation*}
Here $M^\vee=\{x\in N\mid \langle x, M\rangle\subset W \}$ denotes the dual lattice of $M$ w.r.t. $\langle \,,\, \rangle$, and the symbol $n$ resp. $\leq 1$ over 
the inclusion sign means that the quotient is a $\BF$-vector space of dimension  $n$, resp. $\leq 1$.
\end{proposition}
\begin{proof}
 Let $(X, \iota, \lambda, \varrho) \in \bar \CN^0 (\BF)$. The quasi-isogeny $\varrho$ defines the inclusion of the 
Dieudonn\'{e} module of $X$ as a Dieudonn\'{e} lattice $M$ in $N$. Since $\lambda$ is a principal polarization, $M$ is a selfdual lattice. The stability under $\iota (\CO_E)$ is 
equivalent to the stability of $M$ under $\Pi$. Since ${\rm Lie} \, X = M / V M$, the condition \eqref{wedcond} says
that the rank of the endomorphism $\Pi \vert (M / V M)$ is $0$ or $1$. This shows that $M$ satisfies all properties. Conversely, if
$M$ has these properties, it is a Dieudonn\'{e} lattice since automatically $F M \subset M$, and hence $M$ corresponds to a point
in $\bar \CN^0 (\BF)$.
\end{proof}
Since $N$ is supersingular, all slopes of the $\sigma$-linear operator $ \Pi \, V^{-1} : N \lra N$ are zero. Let $\epsilon\in \BZ_p^\times$ such that $\pi_0=\epsilon p$, and let $\eta\in W^\times$ such that $\eta^2=\epsilon^{-1}$. The $\sigma$-linear operator $\tau=\eta\,\Pi\, V^{-1}: V\to V$ has all slopes equal to zero. 
Setting $C = N^\tau$, we obtain a $\BQ_p$-vector space with an isomorphism
\begin{equation*}
 C \otimes_{\BQ_p}W_\BQ \tilde{\lra} N
\end{equation*}
such that ${\rm id}_C \otimes \sigma$ corresponds to $\tau$.
Then for $x,y \in C$
\begin{equation*}
 \begin{aligned}
  \langle x,y \rangle = \langle \tau (x), \tau (y)\rangle 
& = \langle \eta\,\Pi \, V^{-1}x, \eta\,\Pi \, V^{-1}y\rangle = \epsilon^{-1} \pi_0 \langle V^{-1} x, V^{-1} y \rangle \\
& = \epsilon^{-1} \pi_0  \cdot p^{-1} \cdot \langle x,y \rangle^\sigma\\
& = \langle x,y \rangle^\sigma \, .
 \end{aligned}
\end{equation*}
Hence the restriction of $\langle \,, \, \rangle$ to $C$ defines a skew-symmetric form with values in $\BQ_p$ 
such that
\begin{equation*}
 \langle \Pi \, x,y \rangle= - \langle x, \Pi \, y \rangle  \, , \quad x,y \in C \, .
\end{equation*}
The action of $\Pi$ defines the structure of an $n$-dimensional $E$-vector space on $C$, and
\begin{equation}
 (x,y) := \langle \Pi \, x,y \rangle + \langle x,y \rangle \cdot \pi
\end{equation}
defines a hermitian form on $C$ satisfying
\begin{equation}\label{spur}
 \langle x,y \rangle = \frac{1}{2} {\rm Tr}_{E/\BQ_p} (\pi^{-1}( x,y)) \, .
\end{equation}
From now on we write $\pi$ instead of $\Pi$ for the action on $C$.

Recall from \cite{RZ} the linear algebraic group $J$ over $\BQ_p$, with 
\begin{equation}\label{DefJ}
\begin{aligned}\
 J(R) = \{g\in {\rm GL}_{R \otimes W_\BQ} (N \otimes R) &\mid g F = Fg, g \Pi = \Pi g \, ,\\
 \langle g x, g y \rangle &= c(g) \langle x,y \rangle , c(g) \in (R \otimes W_\BQ)^\times, \forall x,y \} \, ,
 \end{aligned}
\end{equation}
for any $\BQ_p$-algebra $R$.
\begin{lemma}
 The algebraic group $J$ is canonically isomorphic to the group of unitary similitudes ${\rm GU}(C)$ of $(C,(\,,\,))$.
\end{lemma}
\begin{proof}
 We content ourselves with showing that the restriction map $g \mapsto g \vert _C$ induces an isomorphism
\begin{equation*}
 J(\BQ_p) \simeq {\rm GU}(C) (\BQ_p) \, .
\end{equation*}
This map indeed maps $ J(\BQ_p)  $ into $ {\rm GU}(C)(\BQ_p)$ since any $g\in J(\BQ_p)$ commutes with $\tau$ and it follows from the discussion above that $c(g)\in \BQ_p^{\times}$. This and the definition of the hermitian form  imply that  $g \vert _C\in {\rm GU}(C) (\BQ_p) =\{ g\in {\rm GL}_{E} (C) \mid  ( g x, g y ) = \gamma(g) ( x,y ) , \gamma(g) \in \BQ_p ^\times, \forall x,y \} $. Conversely, suppose that $g \in {\rm GU}(C) (\BQ_p)$. We extend it linearly to an endomorphism of $N=C\otimes W_\BQ$ by setting $g(v\otimes a)=a\cdot g(v)$. By construction it commutes with $\pi$. It also commutes with $F$: For $v\in C$ we have $gF(v)=Fg(v)$ since $F=p\pi^{-1}\eta^{-1}$ on $C$. For $v\in C$ and $a\in W_{\BQ}$ we  have $gF(v\otimes a)= a^{\sigma}gF(v)=a^{\sigma}Fg(v)=F(g(v)\otimes a)=Fg(v\otimes a)$. The relation $ \langle g x, g y \rangle = c(g) \langle x,y \rangle $ follows from the corresponding relation in ${\rm GU}(C)$ and \eqref{spur}. This shows that the map is surjective. It also follows that it is injective since the linear extension of an element  $g$ in ${\rm GU}(C) (\BQ_p)$ to an endomorphism of $N$ is unique. 
\end{proof}
We denote by $J(\BQ_p)^o$ the subgroup of $J(\BQ_p)$ where the multiplier $c(g)$ is a unit. Then 
$J(\BQ_p)^o$ acts on $\bar \CN^0$, via
\begin{equation*}
 g : (X, \iota, \lambda, \varrho) \longmapsto (X, \iota, \lambda, g \circ \varrho) \, .
\end{equation*}

We may reformulate the description of $\bar \CN^0 (\BF)$ in Proposition \ref{pointdescr} in terms of the hermitian
vector space $C$ as follows. We extend the hermitian form on $C$ to $C \otimes_E \breve{E}$ by
\begin{equation*}
 (v \otimes a, v' \otimes a') = a \cdot a'^\sigma \cdot (v, v'), \, a, a' \in \breve{E} \, .
\end{equation*}
For a $\CO_{\breve{E}}$-lattice $M$ in $C \otimes_E \breve{E}$ we denote by  $M^\sharp=\{ x\in N\mid (x, M)\subset \CO_{\breve E}\}$ its dual w.r.t. to the form $(\,,\,)$. 
Then $M^\vee = M^\sharp$ for any $\CO_{\breve{E}}$-lattice $M$ in $C \otimes_E \breve{E}$, i.e., the duals w.r.t.
$\langle\, , \,\rangle$ or $(\,,\,)$ coincide.

\begin{proposition}\label{moduldescr}
 There is a bijection between $\bar \CN^0 (\BF)$ and the set of $\CO_{\breve{E}}$-lattices
\begin{equation*}
\begin{aligned}
 \CV (\BF) = \{ M \subset C \otimes_E \breve{E} &\mid M^\sharp = M \, ,\\
\Pi \tau (M) &\subset M \subset \Pi^{-1} \tau (M), \, M \subset^{\leq 1}\big( M + \tau (M)\big) \} \, .
\end{aligned}
\end{equation*}\qed
\end{proposition}

\section{Vertex lattices}\label{section.Vertexlattices}

In the next  section we will construct a decomposition of $\CV (\BF)$  into a disjoint sum of subsets
parametrized by vertex lattices in $C$. In the present section we analyze this concept. For a lattice $\Lambda$ in $C$, we denote by $\Lambda^\sharp=\{x\in C\mid (x, \Lambda)\subset \CO_E\}$ its dual lattice. 
\begin{definition}
 A \emph{vertex lattice} in $C$ is a $\CO_E$-lattice $\Lambda$ such that 
$\pi \Lambda \subset \Lambda^\sharp \subset \Lambda$. The dimension of the $\BF_p$-vector space $\Lambda / \Lambda^\sharp$
is called the \emph{type} of $\Lambda$.
\end{definition}

Let $\Lambda_1$ and $\Lambda_2$ be vertex lattices of type $t_1$, resp. $t_2$. Then $\Lambda_1 \subseteq \Lambda_2$
implies $t_1 \leq t_2$, with equality if and only if $t_1 = t_2$. Furthermore, $\Lambda_1 \cap \Lambda_2$ is
a vertex lattice if and only if $\Lambda_1^\sharp \subset \Lambda_2$ if and only if 
$\Lambda_2^\sharp \subset \Lambda_1$. Indeed,
\begin{equation*}
 (\Lambda_1 \cap \Lambda_2)^\sharp = \Lambda_1^\sharp + \Lambda_2^\sharp
\end{equation*}
is contained in $\Lambda_1 \cap \Lambda_2$ if and only if $\Lambda_1^\sharp \subset \Lambda_2$, and if the last
condition is satisfied, then
\begin{equation*}
 \pi  (\Lambda_1 \cap \Lambda_2) \subset (\Lambda_1 \cap \Lambda_2)^\sharp \, ,
\end{equation*}
hence $\Lambda_1 \cap \Lambda_2$ is a vertex lattice.

\begin{lemma}\label{typeeven}
The type of a vertex lattice  is an even integer.
\end{lemma}
\begin{proof}
We define the index $[L:M]$ of two arbitrary lattices in $C$ as 
$$[L:M]:=[L:L\cap M]-[M:L\cap M] .
$$ Then $[L:M]=-[M:L]$. For a third lattice $N$ we have $[L:M]=[L:N]+[N:M]$. Furthermore $[L:M]=[M^{\sharp}:L^{\sharp}]$.
Now let 
$\Lambda$ be a vertex lattice and let $L$ be any lattice. Then $[\Lambda:\Lambda^{\sharp}]=[\Lambda:L]+[L:\Lambda^{\sharp}]=[\Lambda:L]+[\Lambda:L^{\sharp}]=[\Lambda:L]+[\Lambda:L]+[L:L^{\sharp}]$. If we choose for $L$ a self-dual lattice (which exists since $E/\BQ_p$ is {\it ramified}, cf. \cite{J}) we obtain the claim.
\end{proof}

Recall that the discriminant of the hermitian space  $C$ is  the image of $(-1)^{n(n-1)/2}\det V$ in $\BQ_p^\times/{\rm Nm}_{E/\BQ_p}(E^\times)$. Note that, since $E/\BQ_p$ is ramified, this last group is generated by the units in $\BZ_p$. We call $C$ {\it split} if the discriminant is the trivial element in this group of order $2$, and {\it non-split} if the discriminant is non-trivial.

\begin{lemma}\label{type}
\noindent If $n$ is even and  $C$  is split,  then for 
 all even $t$ with $0\leq t\leq n$ there is a  vertex lattice of type $t$.
 
 \smallskip

\noindent If $n$ is even and $C$ is non-split,  then for 
 all even $t$ with $0\leq t\leq n-2$ there is a  vertex lattice of type $t$, but there is no vertex lattice of type $n$.
 
 \smallskip

\noindent If $n$ is odd,   then for 
 all even $t$ with $0\leq t\leq n-1$ there is a  vertex lattice of type $t$. 
\end{lemma}

\begin{proof}
We denote by $H$ the anti-diagonal matrix in ${\rm M}_2(E)$, with $1$'s  on the anti-diagonal. Let $J$  be the diagonal matrix in ${\rm M}_2(E)$, with $u_1, u_2$  on the diagonal, where $u_1, u_2\in \BZ_p^\times$ are units such that $-u_1u_2\notin {\rm Nm}_{E/\BQ_p}(E^\times)$. Then $H$ defines the split hermitian space of dimension $2$ over $E$, and $J$ defines the non-split space.

If $n$ is even then $C$ is isomorphic to $E^n$ with hermitian form (wrt. the standard basis) given  by either the matrix $\diag(H,...,H)$, consisting of $n/2$ copies of $H$  or by the matrix 
$\diag(H,...,H, J)$ consisting of $n/2-1$ copies of $H$ and one copy of the matrix $J$. The first case occurs if and only if $C$ is split, the second case if and only if $C$ is non-split.
This shows that in the first case there is a vertex lattice of type $n$, namely the lattice spanned by the vectors $\pi^{-1}e_i$, for $i$  odd, and   by the vectors $e_j$ for $j$ even. In the second case,  there is at least a vertex lattice of type $n-2$ namely the lattice spanned by the vectors $\pi^{-1}e_i$ for $i\leq n-2$  odd,  and by $e_j$ for $j$  even or equal to $n-1$.  It follows from \cite{J}, Proposition 8.1 that there is no vertex lattice of type $n$ if $C$ is non-split. In both cases a vertex lattice of even type $t\leq n-2$ is obtained as the span of the $\pi^{-1}e_i$ for $i\leq t$ and odd and the $e_j$ for $j$ even or $j>t$.

If $n$ is odd then, after multiplying the form by a scalar in $\BZ_p^{\times}$ (and hence the discriminant is irrelevant), we may assume that $C$ is isomorphic to  $E^n$ with hermitian form (wrt. the standard basis) given  by  the matrix 
$\diag(H,...,H,1)$ consisting of $(n-1)/2$ copies of $H$ and one copy of the $1\times 1$ unit matrix $1$. It follows as above that there is a vertex lattice of any even type $t\leq n-1$.
 \end{proof}

The set $\CT$ of vertex lattices in $C$  forms a simplicial complex as follows.   We call distinct vertex lattices $\Lambda_1$ and $\Lambda_2$  {\it neighbours}, if $\Lambda_1\subset \Lambda_2$ or $\Lambda_2\subset \Lambda_1$. Then an $r$-simplex is formed by elements $\Lambda_0,\ldots,\Lambda_r$ such that any two members of this set are neighbours. It is clear that the action of $\SU(C)(\BQ_p) = \SU (C, ( \,,\, ))(\BQ_p)$ on $\CT$
preserves this structure of simplicial complex.

 We define another simplicial complex $\CL$ as follows. If $n$ is even and $C$ non-split,  or if $n$ is  odd, then $\CL=\CT$. Next let $n$ be even and $C$ split. Hence the maximal type of a vertex lattice is $n$. In this case we let $\CL$ be the set of vertex lattices that are not of type $n-2$. Then we call distinct elements $\Lambda_1$ and $\Lambda_2$ in $\CL$ neighbours if $\Lambda_1\subset \Lambda_2$ or $\Lambda_2\subset \Lambda_1$. Note that in this case, not both $\Lambda_1$ and $\Lambda_2$ are of type $n$. If $\Lambda_1$ and $\Lambda_2$ are both of type $n$, then we call them neighbours if their intersection is a vertex lattice of type $n-2$. 
It is clear that  the action of $\SU(C)(\BQ_p) = \SU (C, ( \,,\, ))(\BQ_p)$ on $\CL$
preserves this structure of simplicial complex.
\begin{proposition}\label{Gebaeude} A set of vertex lattices $\{ \Lambda_0, \ldots, \Lambda_r \}$ is an  $r$-simplex in $\CT$ if and only if, after renumbering the $\Lambda_i$, there is a chain of inclusions,
\begin{equation*}
 \Lambda_0 \supset \Lambda_1 \supset \ldots \supset \Lambda_r \, .
\end{equation*}
 There is a $\SU (C) (\BQ_p)$-equivariant isomorphism between $\CL$ and the Bruhat-Tits simplicial complex of 
$\SU (C)$ over $\BQ_p$.

The simplicial complexes $\CT$ and $\CL$ are connected. 

\end{proposition}
\begin{proof} The first assertion is trivial. The relation between $\CL$ and the Bruhat-Tits simplicial complex is  discussed in \cite{PR1}, \S 4, in the cases where  $C$ is split, comp. also \cite{PR3}, \S 1.2. The stabilizers of vertex lattices in $\SU(C)(\BQ_p)$ are always  parahoric subgroups. Furthermore, these parahoric subgroups are always maximal, and hence correspond to vertices in the Bruhat-Tits simplicial complex, except when $n$ is even, and the type of the
vertex lattice is $n-2$. In the latter case the corresponding parahoric subgroup is the intersection of the stabilizers of the two vertex lattices of type $n$ containing the vertex lattice of type $n-2$. Hence a vertex lattice of type $n-2$  corresponds to an edge in the Bruhat-Tits complex, with vertices corresponding to the two vertex lattices of type $n$ containing it. This discussion extends to the case when $C$ is non-split, where this last  phenomenon does not occur (use the normal form of the hermitian form explained in the proof of Lemma \ref{type} above). 

It is well-known that the Bruhat-Tits complex is connected; the connectedness of $\CT$ follows. 
\end{proof}

\section{The pointwise Bruhat-Tits stratification of $\bar \CN^0_{\rm red}$}\label{section.BruhatTitsstratification}

In this section we will construct a decomposition of $\CV (\BF)$  into a disjoint sum of subsets
parametrized by vertex lattices in $C$. The subset associated to the vertex lattice $\Lambda$ is called
the \emph{Bruhat-Tits stratum} $\CV_\Lambda^\circ(\BF)$ associated to $\Lambda$, in analogy with \cite{VW}.

For a $\CO_{\breve{E}}$-lattice $M$ in $C \otimes_E \breve{E}$, and a non-negative integer $j$, let
\begin{equation*}
 T_j (M) = M + \tau (M) + \ldots + \tau^j (M) \, .
\end{equation*}
\begin{proposition}\label{pointsvertexlattice}
 Let $M\in \CV(\BF)$ be an $\CO_{\breve{E}}$-lattice in $C \otimes_E\breve{E}$ which corresponds to a point 
of $\bar \CN^\circ (\BF)$ under the bijection of Proposition \ref{moduldescr}. There exists an (unique) integer 
$d \leq \frac{n}{2}$ such that there is a chain of inclusions of $\CO_{\breve{E}}$-lattices
\begin{equation*}
 M \subset^1 T_1(M) \subset^1 \ldots \subset^1 T_d(M) = T_{d+1} (M) = \ldots
\end{equation*}
The intersection $\Lambda (M) = T_d (M) \cap C = T_d (M)^\tau$ is a vertex lattice of type $2d$. It is the 
minimal lattice $\Lambda$ in $C$ such that $\Lambda \otimes_{\CO_E} \CO_{\breve{E}}$ contains $M$. Dually,
$\Lambda^{\sharp}$ is the maximal lattice in $C$ such that $\Lambda^{\sharp} \otimes_{\CO_E} \CO_{\breve E}$ is contained
in $M$.
\end{proposition}
\begin{proof}
 We write $T_j$ for $T_j(M)$. Obviously $T_{j+1} = T_j + \tau(T_j)$. Let $d$ be minimal such that 
$T_d = \tau (T_d)$; in particular, $T_j \neq \tau (T_j)$ for $0 \leq j < d$. Such an integer exists, cf.  \cite{RZ}, Proposition~2.17. 
  We need to show that $T_0 \subset^1 T_1 \subset^1 \ldots \subset^1 T_d = T_{d+1} = \ldots$ and
that $T_d^\tau$ is a vertex lattice of type $2d$ in $C$ (hence also $d \leq n/2)$. If $d = 0$, then all assertions are
obvious, hence we may assume $d \geq 1$.

\emph{Case} $d = 1$. In this case, by Proposition \ref{moduldescr} we have $T_0 \subset^1 T_1$. Furthermore,
\begin{equation*}
 T_1^\sharp \subset^1 T_0^\sharp = T_0 \subset^1 T_1 \, .
\end{equation*}
 We need to show that $\Pi \, T_1 \subset T_1^\sharp$. But
\begin{equation*}
 T_1^\sharp = (T_0 + \tau (T_0))^\sharp = T_0^\sharp \cap \tau (T_0)^\sharp = T_0 \cap \tau (T_0) \, .
\end{equation*}
Since $\Pi \, T_0 \subset \tau (T_0)$ we deduce $\Pi \, T_0 \subset T_0 \cap \tau (T_0)$; since 
$\Pi \, \tau(T_0) \subset T_0$ we deduce $\Pi  \tau (T_0)\subset T_0\cap  \tau(T_0)$. Hence
\begin{equation*}
 \Pi \, T_1 = \Pi \, T_0 + \Pi \, \tau (T_0) \subset T_0 \cap \tau (T_0) = T_1^\sharp \, ,
\end{equation*}
as claimed.

\emph{Case} $d \geq 2$. We use
\begin{equation*}
 M + \tau (M) + \tau^2(M) \subset \Pi^{-1} \, \tau(M) \, .
\end{equation*}
Hence, since $d \geq 2$,
\begin{equation*}
 \begin{aligned}
  T_d 
& = \big(M+ \tau (M) + \tau^2 (M)\big) + \ldots + \tau^{d-2} \big(M + \tau (M) + \tau^2 (M)\big)\\
& \subseteq \Pi^{-1} \tau (M) + \Pi^{-1} \tau^2 (M) + \ldots + \Pi^{-1} \tau^{d-1} (M)\\
& = \Pi^{-1} T_{d-1} \, .
 \end{aligned}
\end{equation*}
Hence, since $\tau (T_d) = T_d$  ,
\begin{equation*}
 T_d \subseteq \Pi^{-1} \cdot \bigcap_{\ell \in \BZ} \tau^\ell (T_{d-1}) \, .
\end{equation*}
Furthermore, for $j = 0, \ldots, d - 1$, we have $T_j \subset^{1} T_{j+1}$. Hence 
$T_j \cap \tau (T_j) = \tau (T_{j-1})$ for $j = 1, \ldots, d - 1$, because both lattices are contained
in $\tau (T_j)$ with index $1$. By induction we get
\begin{equation*}
 \bigcap_{\ell \in \BZ} \tau^\ell (T_{d-1}) = \bigcap_{\ell \in \BZ} \tau^\ell (M) .
\end{equation*}
Hence
\begin{equation*}
 \begin{aligned}
  T_d 
& \subset \Pi^{-1} \bigcap \tau^\ell (M)\\
& = \Pi^{-1} \bigcap \tau^\ell (M^\sharp)\\
& = \Pi^{-1} \bigcap \tau^\ell (M)^\sharp\\
& \subseteq \Pi^{-1} \bigcap_{0 \leq \ell \leq d} \tau^\ell (M)^\sharp\\
& = \Pi^{-1} \cdot T_d^\sharp \, .
 \end{aligned}
\end{equation*}
Hence we obtain the chain of inclusions
\begin{equation*}
 \Pi \,  T_d \subset T_d^\sharp \subset^d M^\sharp = M \subset^d T_d \, .
\end{equation*}
It follows that $\Lambda = T_d^\tau$ is a vertex lattice in $C$ of type $2 d$. In particular, as mentioned
above, $2d \leq n$.
\end{proof}

\begin{remark}
For even $n$, indeed both possibilities, for $C$ to be split and to be non-split, 
 can occur depending on the choice of $(\BX, \iota, \lambda_\BX)$ used to define the moduli problem for $\CN$. For $n=2$ this is discussed in \cite{KR3}, section 5. Let $\CE$ be the $p$-divisible group of a supersingular elliptic curve over $\BF$, with its action by $\CO_E$ (obtained by embedding $E$ into the quaternion algebra $\End^0(\CE)$), and with its natural principal polarization, which allows us to identify $\CE$ with $\CE^\vee$. Let $\BX_2^{+}$ be the $p$-divisible group $\CE\times\CE$ with the diagonal action by $\CO_E$, and the polarization given by the anti-diagonal matrix $H$ with $1$'s on the anti-diagonal, cf. proof of Lemma \ref{type}. The associated hermitian space $C$ is split, cf. \cite{KR3}. Similarly, let  $\BX_2^{-}$ be defined in the same way as $\BX_2^{+}$, except that the polarization is defined by the diagonal matrix with entries $u_1, u_2\in \BZ_p^\times$ with $-u_1u_2\notin {\rm Nm}_{E/\BQ_p}(E^\times)$. The associated hermitian space $C$ is non-split.
 Then for any even $n$, we can construct a triple  $(\BX, \iota, \lambda_\BX)$ with  $C$ split by choosing $\BX$ to be the product of $n/2$ copies of $\BX_2^{+}$ and for the action $\iota$ and the polarization $\lambda_\BX$ of $\BX$ the corresponding product action resp. product polarization. Similarly 
for any even $n$, we can construct a triple  $(\BX, \iota, \lambda_\BX)$ with $C$ non-split  by choosing $\BX$ to be the product of $n/2-1$ copies of $\BX_2^{+}$ and one copy of $\BX_2^{-}$ and for the action $\iota$ and the polarization $\lambda_\BX$ of $\BX$ the corresponding product action resp. product polarization. 
Conversely, for $n$ even,  the choice of a discriminant determines $(\BX, \iota, \lambda_\BX)$ up to isogeny.  

For $n$ odd,  $(\BX, \iota, \lambda_\BX)$ is unique up to isogeny since we may multiply $\lambda_\BX$ by elements in $\BZ_p^{\times}$ which makes the discriminant of $C$ in this case irrelevant. (Compare also the proof of \cite{VW}, Lemma 6.1.)
\end{remark}
Let $\Lambda$ be a vertex lattice in $C$, and set
\begin{equation*}
 \CV_\Lambda (\BF) = \{ M \in \CV (\BF) \mid M \subset \Lambda \otimes_{\CO_E} \CO_{\breve{E}} \} \, .
\end{equation*}
Then by the previous proposition,
\begin{equation}
 \CV (\BF) = \bigcup\nolimits_\Lambda \CV_\Lambda (\BF) \, .
\end{equation}
More precisely, let
\begin{equation*}
 \CV_\Lambda^\circ (\BF)= \{ M \in \CV (\BF) \mid \Lambda (M) = \Lambda \} \, .
\end{equation*}
Then $ \CV_\Lambda^\circ (\BF) \subset \CV_\Lambda (\BF)$, and 
\begin{equation}
\CV (\BF) = \biguplus\nolimits_\Lambda \CV_\Lambda^\circ (\BF) ,
\end{equation}
(disjoint union).

\begin{proposition}\label{pointstrat}
 Let $\Lambda_1$ and $\Lambda_2$ be vertex lattices in $C$.
 
 \smallskip
 
\noindent (i) $\CV (\Lambda_1) \subseteq \CV (\Lambda_2)$ if and only if $\Lambda_1 \subseteq \Lambda_2$, and equality
holds if and only if $\Lambda_1 = \Lambda_2$.
Hence
\begin{equation*}
 \CV_\Lambda (\BF) = \biguplus\nolimits_{\Lambda' \subset \Lambda} \CV_{\Lambda'}^\circ (\BF)\, ,
\end{equation*}
and all summands are non-empty.

\smallskip

\noindent (ii) If $\Lambda_1 \cap \Lambda_2$ is a vertex lattice, then 
\begin{equation*}
 \CV_{\Lambda_1 \cap \Lambda_2}(\BF) = \CV_{\Lambda_1}(\BF) \cap \CV_{\Lambda_2}(\BF) \, .
\end{equation*}
Otherwise $\CV_{\Lambda_1}(\BF) \cap \CV_{\Lambda_2}(\BF) = \emptyset$.
\end{proposition}
\begin{proof}
 (i) The first assertion is obvious. 
The second claim  follows 
 from Propositions \ref{decomp} and \ref{iso} below. 

\noindent (ii) If $\Lambda = \Lambda_1 \cap \Lambda_2$ is a vertex lattice, it is obvious that
\begin{equation*}
 \CV_\Lambda (\BF) = \CV_{\Lambda_1}(\BF) \cap \CV_{\Lambda_2}(\BF) = \{ M \in \CV (\BF) \mid
M \subset (\Lambda_1 \otimes_{\CO_E} \CO_{\breve{E}}) \cap (\Lambda_2 \otimes_{\CO_E} \CO_{\breve{E}}) \} \, . 
\end{equation*}
 If $\CV_{\Lambda_1}(\BF) \cap \CV_{\Lambda_2}(\BF) \neq \emptyset$, let $M \in \CV_{\Lambda_1}(\BF) \cap \CV_{\Lambda_2}(\BF)$
lie in $\CV _\Lambda^\circ(\BF)$. Then $\Lambda \subset  \Lambda_1 \cap \Lambda_2$ by the maximality of 
$\Lambda = \Lambda (M)$. But then $\Lambda_1^\sharp +\Lambda_2^\sharp\subset \Lambda^\sharp \subset \Lambda \subset \Lambda_1\cap \Lambda_2$,
hence $\Lambda_1 \cap \Lambda_2$ is a vertex lattice.
\end{proof}

\section{The Deligne-Lusztig variety for the symplectic group}\label{section.DLvariety}
In this section, we analyze some (generalized) Deligne-Lusztig varieties. Their relevance to the general goal of the paper will become clear in the next section.

In the present section we adopt the point of view of the original paper \cite{DL} of Deligne and Lusztig; in particular, we describe the algebraic varieties that occur here by their points over $\bar\BF_p$.

Let $G$ be a split reductive group over a finite field $\BF_q$. (Later, $G$ will be a symplectic group over
$\BF_p$.) Let $T$ be a split maximal torus, $B$ a Borel subgroup containing $T$ and let 
$\Delta^* = \{ \alpha_1, \ldots , \alpha_n \}$ be the set of simple roots corresponding to $(T, B)$. 
Let $s_i = s_{\alpha_i}$ be the corresponding simple reflections in the Weyl group $W$. For 
$I \subset \Delta^*$, let $W_I$ be the subgroup of $W$ generated by $\{ s_i \mid i \in I \}$, and let
$P_I = BW_IB$ denote the corresponding standard parabolic subgroup. The quotient $G/P_I$ parametrizes the 
parabolic subgroups of type $I$. We denote by
\begin{equation*}
 {\rm inv}: G/P_I \times G/P_I \lra W_I \backslash W/W_I
\end{equation*}
 the \emph{relative position morphism}. For $w \in W_I \backslash W/W_I$, let
\begin{equation*}
 \CO_{P_I}(w) = \{ (P_1, P_2) \in G/P_I \times G/P_I \mid {\rm inv} (P_1, P_2) = w \} \, .
\end{equation*}
The \emph{generalized Deligne-Lusztig variety} $X_{P_I}$ corresponding to $w \in W_I \backslash W/W_I$ is the intersection
of $\CO_{P_I} (w)$ with the graph of Frobenius in $(G/P_I) \times (G/P_I)$, i.e., is the locally closed subscheme of $G/P_I$ consisting of all parabolic subgroups of type $I$ which are in relative position $w$ wrt its translate under $\Phi$.
 Here, and in the sequel, $\Phi$ denotes the Frobenius endomorphism. For $I = \emptyset$, we have $P_I = B$ and $X_B(w)$ is the usual Deligne-Lusztig variety attached to $w \in W$.

Since $\CO_{P_I} (w)$ is a homogeneous space under $G$, it is a smooth variety. The same follows for $X_{P_I}(w)$, which is smoothly equivalent to $\CO_{P_I} (w)$ by the {\it G\"ortz  local model diagram} in the context of Deligne-Lusztig varieties, comp. \cite{GY}, 5.2. Hence  $X_{P_I}(w)$ is equi-dimensional, with
\begin{equation}\label{dimDL}
 {\rm dim}\, X_{P_I} (w) = \ell_I(w) - \ell(w_I) \, ,
\end{equation}
where $\ell_I(w)$ denotes the maximal length of an element in the double coset $W_I w W_I$, and where $w_I$ 
denotes the longest element in $W_I$. In particular, for $I = \emptyset$, we have ${\rm dim}\, X_B (w) = \ell(w)$.

Now let $I = \emptyset$. For $w \in W$, let $\overline{\CO_B (w)}$ be the corresponding Schubert variety, i.e.
the reduced closure of $\CO_B (w)$ in $(G/B) \times (G/B)$. By the definition of the Bruhat order, there
is a stratification
\begin{equation*}
 \overline{\CO_B (w)} = \biguplus_{w' \leq w} \CO_B (w') \, .
\end{equation*}
We recall the Demazure resolution of $\overline{\CO_B (w)}$. Fix a reduced decomposition 
$w = s_{i_1}\cdot \ldots\cdot s_{i_m}$ for $w$, and consider the closed subscheme of $(G/B)^{m+1}$,
\begin{equation*}
 \bar{\CO}_B (s_i, \ldots, s_{i_m}) = \{ (B_0, \ldots, B_m) \mid {\rm inv} (B_{j-1}, B_j) 
\leq s_{i_j} \, , j = 1, \ldots, m \} \, .
\end{equation*}
Under the projection to the $0$-th and the $m$-th component, we obtain a morphism
\begin{equation*}
 p : \bar\CO_B (s_i, \ldots, s_{i_m}) \lra \overline{\CO_B (w)} \, .
\end{equation*}
\begin{proposition}
 If $w$ is a Coxeter element, $p$ is an isomorphism. In particular, $\overline{\CO_B (w)}$ is smooth, and the
complement of $\CO_B (w)$ in $\overline{\CO_B (w)}$ is a divisor with normal crossings.
\end{proposition}
\begin{proof}
As is well known, $p$ is birational. 
 The bijectivity of $p$ is proved in \cite{H}, Lemma 1. Since $\overline{\CO_B (w)}$ is normal, $p$ is an
isomorphism by Zariski's Main Theorem. The remaining assertions follow from properties of Demazure 
resolutions.
\end{proof}
\begin{corollary}\label{gencox}
 Let $w$ be a Coxeter element. The closure $\overline{X_B (w)}$ of $X_B (w)$ in $G/B$ is smooth and its
complement is a divisor with normal crossings. It is stratified by Deligne-Lusztig varieties for smaller
elements than $w$ in the Bruhat order,
\begin{equation*}
 \overline{X_B (w)} = \biguplus_{w' \leq w} X_B (w') \, .
\end{equation*}
Furthermore, $\overline{X_B (w)}$ is irreducible.
\end{corollary}
\begin{proof}
 The first assertion follows from the previous proposition, again   by the {\it G\"ortz  local model diagram} in the context of Deligne-Lusztig varieties, comp. \cite{GY}, 5.2. This implies that $ \overline{X_B (w)} $ is smoothly equivalent to $\overline{\CO_B (w)}$, both equipped with their natural stratification. The second assertion is well-known, cf., e.g.,   \cite{G}.
\end{proof}
From now on we consider the symplectic group. Let us fix notation as follows. Let $(V, \langle\,, \,\rangle)$ be 
a symplectic vector space of dimension $n=2 m$ over $\BF_p$. We choose a basis $(e_1, \ldots, e_{2 m})$ of $V$ such that
\begin{equation*}
 \langle e_i, e_{2m + 1-j} \rangle = \pm \, \delta_{i,j} \, ,\,\, i, j=1, \ldots, 2m\, .
\end{equation*}
Let $T \subset {\rm Sp} (V)$ be the diagonal torus, and $B \subset {\rm Sp} (V)$ the group of upper triangular 
matrices. Then the simple reflections in $W$ can be enumerated as follows:

\smallskip

\noindent$\bullet$ for $1 \leq i \leq m-1$, the reflection $s_i$ permutes $e_i$ and $e_{i+1}$, and also $e_{2m + 1-i}$
and $e_{2m-i}$, 

\smallskip

\noindent$\bullet$  $s_m$ permutes $e_m$ and $e_{m+1}$.

\smallskip

\noindent We consider the Coxeter element $w = s_1 s_2 \cdot \, \ldots \, \cdot s_m$. It sends 
$e_i \longmapsto e_{i+1}$ for $1 \leq i \leq m-1$, and $e_m \longmapsto e_{2m}$, and $e_{m+1} \longmapsto e_1$, 
and $e_i \longmapsto e_{i-1}$ for $m+2 \leq i \leq 2m$. 

The corresponding Deligne-Lusztig variety $X_B(w)$ can be identified with the variety of complete isotropic flags
\begin{equation*}
 (0) = F_{0} \subset F_{1} \subset \ldots \subset F_m
\end{equation*}
such that for $i=0, \ldots, m-1$
\begin{equation*}
 F_{i} = F_{i+1} \cap \Phi (F_{i+1}) \, .
\end{equation*}
(Note that our normalization of the relative position map is the opposite of the one of Deligne and Lusztig 
\cite{DL}; in their notation, we are considering $X_B(w^{-1})$). Associating to such a flag its $m$-th member,
we obtain an identification
\begin{equation}\label{descrDL}
\begin{aligned}
 X_B (w) = & \{ U \mid U \, {\rm Lagrangean \, subspace \, of }\, V \, {\rm such \, that}\\
& {\rm dim} \, \big(U \cap \Phi (U) \cap \ldots \cap \Phi^i (U)\big) = m - i,   0 \leq i \leq m \} \, .
\end{aligned}
\end{equation}
Also, the closure $\overline{X_B(w)}$ in $G/B$ can be characterized by the conditions
\begin{equation*}
 {\rm dim} (F_{m-i} \cap \Phi (F_{m-i})) \geq m-i-1, \,\,\, \, i = 0, \ldots, m-1 \, .
\end{equation*}
For $0 \leq i \leq m$, let
\begin{equation*}
 w_i = s_{m+1-i} \ldots s_m \, .
\end{equation*}
Hence $w_0 = {\rm id}$, and $w_1 = s_m$, and $w_m = w$.

Let $P$ be the Siegel parabolic, i.e., $P$ stabilizes the Lagrangean subspace spanned by $\{ e_1, \ldots, e_m \}$.
Then $G/P$ parametrizes the Lagrangean subspaces of $V$. We consider the subvariety  $S = S_V$ of $G/P$ given by
\begin{equation}\label{SV}
 S =S_V= \{ U \mid U \, {\rm Lagrangean,} \, \dim (U \cap \Phi (U)) \geq m-1 \} \, .
\end{equation}
\begin{proposition}\label{SDLV}
 $S$ can be identified with the closure of the generalized Deligne-Lusztig variety $X_{P}(w_1)$ in $G/P$. In particular,
$S$ is a normal variety with isolated singularities.
\end{proposition}
\begin{proof}
 If $U \in S $,  then ${\rm inv}(U, \Phi (U))$ is either equal to the identity or
equal to $s_m = w_1$ in $W_0 \bs W/W_0$. Hence we obtain a disjoint decomposition
\begin{equation*}
 S = X_{P}({\rm id}) \biguplus X_{P} (w_1) \, ,
\end{equation*}
where $X_{P}({\rm id}) = X_{P} (w_0)$ is a closed and $X_{P} (w_1)$ is an open subset, which proves the
first claim. The second claim follows from the fact that generalized Schubert varieties are normal, and
$\overline{X_{P} (w_1)}$ is smoothly equivalent to $\overline{\CO_{P} (w_1)}$, cf. \cite{GY}.
\end{proof}
\begin{remark}\label{Ssing}
 It follows from \cite{BP}, Proposition~4.4, that for $m \geq 2$, $S$ is singular at the points in 
$S \setminus X_{P} (w_1)$, i.e., in $X_{P}({\rm id})$. For $m = 1$, $S$ is isomorphic to $\BP^1$.
\end{remark}
We next construct a stratification of $S$.
Let, for $i = 0, \ldots, m - 1$,
\begin{equation*}
I_i = \{ 1, \ldots, m - 1 - i \} \, ,
\end{equation*}
and let $W_i = W_{I_i}$ and $P_i = BW_iB$ be the corresponding standard parabolic subgroup. Hence 
$P_0 = P$ and $P_{m-1} = B$. In the next statement we let $P_m = P_{m-1} = B$.
\begin{proposition}\label{decomp}
 There is a decomposition of $S$ into the disjoint union of locally closed subvarieties
\begin{equation*}
 S = \biguplus_{i = 0}^m X_{P_i} (w_i)\, ,
\end{equation*}
such that for every $j$ with $0 \leq j \leq m$, the subset $\biguplus_{i = 0}^j X_{P_i} (w_i)$ is closed in
$S$. The variety $X_{P_m} (w_m) = X_B(w)$ is open in $S$, and irreducible of dimension $m$. In particular,
$S$ is irreducible of dimension $m$.
\end{proposition}
\begin{proof}
 As in the proof of Proposition \ref{SDLV} we see that there is a disjoint decomposition
\begin{equation*}
 S = X_{P_0} ({\rm id}) \biguplus X_{P_0} (w_1) \, ,
\end{equation*}
where $X_{P_0} ({\rm id}) = X_{P_0} (w_0)$ is a closed and $X_{P_0} (w_1)$ is an open subset.

We will show by induction that, for $1 \leq j \leq m - 1$,
\begin{equation*}
 X_{P_{j -1}} (w_j) = X_{P_j} (w_j) \biguplus X_{P_j} (w_{j +1}) \, ,
\end{equation*}
induced from the natural morphism
\begin{equation*}
 \pi : G/P_j \lra G/P_{j-1} \, ,
\end{equation*}
that associates to the isotropic partial flag of length $j + 1$
\begin{equation*}
 F_{m-j} \subset F_{m-(j-1)} \subset \ldots \subset F_m \,,\, \,\,\dim F_{i} =  i
\end{equation*}
the partial flag of length $j$, where the first member is ignored. Since $w_{j + 1} = s_{m-j} \cdot w_j$ and
$\{ s_{m-j} \} = {I}_{j-1}\setminus{ I}_j$, it is clear that
\begin{equation}\label{invim}
 \pi^{-1} (X_{P_{j-1}} (w_j)) = X_{P_j} (w_j) \biguplus X_{P_j} (w_{j + 1}) \, ,
\end{equation}
where the first summand is closed in the LHS and the second summand is open. We have to see that the
restriction of $\pi$ to \eqref{invim} induces an isomorphism with $X_{P_{j-1}} (w_j)$. We claim that the inverse
morphism is given by
\begin{equation*}
 (F_{m-(j-1)} \subset \ldots \subset F_m) \mapsto (F_{m-j} \subset F_{m-(j-1)} \subset \ldots \subset F_m) \, ,
\end{equation*}
with
\begin{equation*}
 F_{m-j} = F_{m-(j+1)} \cap \Phi (F_{m-(j+1)}) \, .
\end{equation*}
Indeed, $\Phi (F_{m-j}) \neq F_{m-j}$  because otherwise $w_j$ would have to lie in
$I_{j-1}$,  which is absurd. On the other hand, the codimension of $F_{m-(j-1)} \cap \Phi (F_{m-(j-1)})$ in
$F_{m-(j-1)}$ is equal to one. Indeed, one checks inductively that
\begin{equation*}
 F_{m-i} = F_{m-(i-1)} \cap \Phi (F_{m-(i-1)}) \,,\,\,\, i = 1, \ldots, j-1 \, .
\end{equation*}
Hence
\begin{equation*}
 F_{m-(j-1)} \cap \Phi (F_{m-(j-1)}) = F_{m-(j-2)} \cap \Phi (F_{m-(j-2)}) \cap \Phi^2 (F_{m-(j-2)})
\end{equation*}
has codimension  $\leq 2$ in $F_{m-(j-2)}$, which proves the claim.

It is clear by construction that for any $j$ with $0 \leq j \leq m$, the subset 
$\biguplus_{i=0}^j X_{P_i} (w_i)$ is closed in $S$. The irreducibility of $X_B(w)$ is a general
fact on Deligne-Lusztig varieties attached to Coxeter elements, cf. Corollary \ref{gencox}.
\end{proof}

\noindent In the sequel we will sometimes write $S_i$ for the locally closed subschema $X_{P_i}(w_i)$ of $S$.

\begin{proposition}\label{dimS}
 For any $i$ with $0 \leq i \leq m$, $S_i$ has pure dimension $i$.
\end{proposition}
\begin{proof}
 By \eqref{dimDL}, the dimension of $X_{P_i} (w_i)$ is given by the difference between the length of the longest
element of $W_iw_iW_i$ and the length of the longest element of $W_i$. Let $x_i$ be the longest element of
$W_i$. We claim that $w_i \cdot x_i$ is the longest element of $W_i \cdot w_i \cdot W_i$. This will prove the claim
since $\ell (w_i x_i) = i + \ell (x_i)$. Let $i \leq m - 1$. Now $w_i = s_{m-i+1}\cdot \ldots \cdot s_m$ commutes with
all elements of $W_i = \langle s_1, \ldots, s_{m-i-1} \rangle$. Hence any shortest expression of an element
in $W_i w_i W_i$ can be arranged so that all reflections in $I_i$ are at the end, and so the longest 
such element is indeed $w_i \cdot x_i$.
The case $i = m$ is trivial.
\end{proof}
The next proposition shows that the stratification of $S = S_V$ has good hereditary properties.
\begin{proposition}\label{irred}
  Let $0 \leq i \leq m \, .$
  
  \smallskip
  
\noindent (i) The irreducible components of $X_{P_i}(w_i)$ are indexed by the isotropic subspaces $W$ of $V$ of dimension
$m - i$.

\smallskip

\noindent (ii) Let $W$ be such an isotropic subspace. The irreducible component $X_{W}$ of $X_{P_i}(w_i)$ corresponding via (i) 
to $W$ is isomorphic to the Deligne-Lusztig variety for the symplectic group $\Sp (W^\vee/W)$ of size $2i$ and its standard
Coxeter element $w' = s'_1\cdot \ldots\cdot s'_i$. The closure of $X_W$ in $S_V$ is isomorphic to $S_{W^\vee/W}$, the variety
defined in the same way as $S_V$, but for the symplectic space $W^\vee/W$.
\end{proposition}
\begin{proof}
(i) A point in $X_{P_i}(w_i)$ corresponds to a partial isotropic flag
\begin{equation*}
 F_{m-i} \subset F_{m-(i-1)} \subset \ldots \subset F_m
\end{equation*}
with $\Phi (F_{m-i}) = F_{m-i}$. Hence $F_{m-i}$ corresponds to an $\BF_p$-rational isotropic subspace $W$ of 
dimension $m - i$, and the subvariety of points of $X_{P_i}(w_i)$ with $F_{m-i} = W$ can be identified with
the Deligne-Lusztig variety for $\Sp (W^\vee/W)$ and the standard Coxeter element. Since this Deligne-Lusztig
variety is irreducible, this proves (i), and the first half of (ii).

Now
\begin{equation*}
 X_{W} = \{ U \in S_V \mid U \cap \Phi (U) \cap \ldots \cap \Phi^i (U) = W \} \, ,
\end{equation*}
and hence its closure is
\begin{equation*}
 \overline{{X}_{W}} = \{ U \in S_V \mid W \subset U \} \, ,
\end{equation*}
which can be identified with $S_{W^\vee/W}$ by associating to $U$ its image in $W^\vee/W$.
\end{proof}
The variety $S$ is singular, comp. Remark \ref{Ssing}. The next proposition constructs a resolution of singularities
in the weak sense.
\begin{proposition}
 Consider the morphism $f: G/B \lra G/P$, which associates to a complete isotropic flag
$\CF= \big((0) = F_0  \subset F_{1} \subset \ldots \subset F_m\big)$ its last member
$F_m$. This morphism induces a surjective proper morphism 
\begin{equation*}
 \pi : \overline{X_B(w)} \lra S\, ,
\end{equation*}
which induces an isomorphism between $X_B(w)$ and the $m$-th stratum $S_m = X_{P_m} (w_m)$ in the stratification
of $S$ in Proposition \ref{decomp}.
\end{proposition}
\begin{proof}
 For $\CF$ in $\overline{X_B(w)}$ it is clear that $F_m \cap \Phi (F_m)$ has  dimension  at least $m - 1$,
i.e., $f$ indeed induces the morphism $\pi$. Now the stratum $S_m$ consists precisely of the Lagrangian 
subspaces $U$ such that
\begin{equation*}
 {\rm dim}\, \big(U \cap \Phi (U) \cap \ldots \cap \Phi^i (U)\big) = m - i \, , \,\,\, 0 \leq i \leq m \, .
\end{equation*}
That $\pi$ induces an isomorphism between $X_B(w)$ and $S_m$ therefore follows from the description of 
$X_B(w)$ in \eqref{descrDL}. The properness of $\pi$ follows from the fact that $\overline{X_B(w)}$ and $S$
are projective varieties, the surjectivity follows because $S_m$ is dense in $S$.
\end{proof}
\begin{remark}
 We note that $\pi$ is (outside the case $m = 1$) only a {\it weak resolution} of singularities, since it is not
an isomorphism above the smooth locus $X_P(s_m)$ of $S$. More precisely, denote by
\begin{equation*}
 D_i = \overline{X_B(s_1 \cdot \, \ldots \, \cdot \widehat{s_i} \cdot \, \ldots \, \cdot s_m)} \quad , \quad i = 1, \ldots, m
\end{equation*}
the $m$ divisors at the boundary of $\overline{X_B(w)}$. Then the inverse image of the $i$-dimensional stratum
$S_i = X_{P_i}(w_i)$ of $S$ is equal to $D_{m-i} \setminus D_{m-i+1}$, for $i = 0, \ldots, m-1$. In particular,
the inverse image of the singular locus $S_0 = X_{P_0} (w_0)$ is the divisor $D_m$. Indeed, for $w' \leq w$, let $\ell$
be the maximal index such that $s_{\ell}$ does not occur in $w'$. It is easily seen that the stratum $X_B(w')$ of 
$\overline{X_B(w)}$ is mapped to $S_{m-\ell}$, which proves the claim.

We note that the morphism $\pi : \overline{X_B(w)} \lra S$ is analogous to the Demazure resolution of
Schubert varieties, resp. of closures of Deligne-Lusztig varieties.
\end{remark}

From now on, we denote by the same symbol $S_V$ the $\BF$-scheme obtained by extension of scalars from the variety considered above.

\section{Bruhat-Tits-strata}\label{section.BTstrata}
In this section we associate to each vertex lattice $\Lambda \subseteq C$ a closed subscheme $\CN_{\Lambda}$ of $\CN$ with $\CN_{\Lambda}(\BF)=\CV_{\Lambda}(\BF)$. The construction imitates that of section  4 of \cite{VW}.

\noindent If   $\Lambda \subset C$ is a lattice and $k$ is a perfect field containing $\BF$ we write $\Lambda_k$ for $\Lambda \otimes_{\BZ_p} W(k)$. Denote by $W(k)_{\BQ}$ the fraction field of $W(k)$. Then we extend the hermitian form $(\, , \, )$ on $C$ to a sesqui-linear form on $C\otimes_{\BQ_p} W(k)_{\BQ}$ setting $(v\otimes a, w\otimes b)=ab^\sigma(v,w)$.
Similarly we extend the alternating form $\langle\, ,\,  \rangle$ linearly to an alternating form on $C\otimes_{\BQ_p} W(k)_{\BQ}$ and will call this form also $\langle\,,\, \rangle$.
 We denote by $M^{\sharp}$ the dual of a lattice $M$ in $C\otimes_{\BQ_p} W(k)_{\BQ}$ with respect to $(\, ,\, )$ and by $M^{\vee}$ the dual with respect to $\langle\, ,\, \rangle$. We use both notations simultaneously, even though $M^\vee=M^\sharp$, as pointed out before Proposition \ref{moduldescr}. 
We fix a vertex lattice  $\Lambda \subseteq C$  of type $t(\Lambda)=2d$. Let $\Lambda^+=\Lambda_{{\BF}}$ and  $\Lambda^-=(\Lambda^{\sharp})_{{\BF}}\supseteq \pi\Lambda_{{\BF}}$.

\begin{lemma} \label{dieud}
Let $k$ be an algebraically closed field of characteristic $p$.
Let $M\subseteq \Lambda_k$ be a $W(k)$-lattice such that $M^{\vee}\subseteq M$. Then $M$ and $M^{\vee}$ are stable under $F, V$ and $\Pi$.
\end{lemma}

\begin{proof} We have 
$FM=pV^{-1}M=\Pi^2V^{-1}M=\Pi\tau M \subseteq \Pi\tau \Lambda_k=\Pi \Lambda_k \subseteq \Lambda_k^{\vee} \subseteq M^{\vee}\subseteq M$. 
Further, $VM=\tau^{-1}\Pi M= \Pi \tau^{-1} M \subseteq \Pi \tau^{-1} \Lambda_k=\Pi \Lambda_k \subseteq M$ and $\Pi M \subseteq \Pi \Lambda_k \subseteq M$. This confirms the claim for  $M$.

For $x\in M^{\vee}$ and any $y \in M$ we have $\langle Fx,y\rangle =\langle x,Vy\rangle ^{\sigma}$ which is integral. Thus $Fx\in M^{\vee}$. Similarly we see that $M^{\vee}$ is stable under $V$ and $\Pi$.
\end{proof}

By the lemma, both  $\Lambda^+$ and $\Lambda^-$ are stable under $V,F$ and $\Pi$. 
In particular, we obtain $p$-divisible $\BZ_p[\pi]$-modules  $X_{\Lambda^+}$ and $X_{\Lambda^-}$ over $\BF$ together with quasi-isogenies $\rho_{\Lambda^+}:X_{\Lambda^+}\rightarrow \BX$ and $\rho_{\Lambda^-}:X_{\Lambda^-}\rightarrow \BX$. Since $\Lambda^-=\Lambda^{+ \sharp}=\Lambda^{+ \vee}$ (compare section \ref{section.Themodulispace}) and hence $X^{\vee}_{\Lambda^+}\cong X_{\Lambda^-}$, we obtain an isomorphism $\lambda: X_{\Lambda^+}\rightarrow X_{\Lambda^-}$ such that the following diagram commutes,

\begin{equation}\label{polcomm}
\begin{aligned}
\xymatrix{ {X_{\Lambda^+}} \ar[d]_{\rho_{\Lambda^+}} \ar[r]^{\lambda}&  {X_{\Lambda^-}}&
 \\ \BX\ar[r]_{\lambda_{\BX}} & \BX^\vee \ar[u]_{^t\rho_{\Lambda^-}} \, .}
 \end{aligned}
\end{equation}

Let $\widetilde{\CN}_{\Lambda}$ be the subfunctor of $\bar\CN^0$ consisting of those tuples $(X,\iota,\lambda,\rho)$ over an $\BF$-scheme $S$ such that $\rho_{\Lambda^+,X}:=(\rho_{\Lambda^+})_S^{-1}\circ \rho$ is an isogeny or, equivalently, such that 
$\rho_{\Lambda^-,X}:=\rho^{-1}\circ(\rho_{\Lambda^-})_S$ is an isogeny. Note that by the commuting diagram (\ref{polcomm}), $\rho_{\Lambda^-,X}$ and $\rho_{\Lambda^+,X}$ both have height $t(\Lambda)/2=d$.

\begin{lemma}\label{rep}
The functor  $\widetilde{\CN}_{\Lambda}$ is representable by a projective $\BF$-scheme and the functor monomorphism $\widetilde{\CN}_{\Lambda}\hookrightarrow\bar\CN^0$ is  a closed immersion.
\end{lemma}
\begin{proof}
This is proved in the 
same way as Lemma 4.2 in \cite{VW}.
\end{proof}






Recall from Proposition \ref{moduldescr} that by associating to a point $x\in \bar\CN^0(k)$ (where $k$ is an algebraically closed field of characteristic $p$) its Dieudonn\'e module, the set $\bar\CN^0(k)$ is in bijection with the set of $W(k)$-lattices 
$$\{ \ M\subseteq N_k :=N\otimes_{W(\BF)} W(k)  \ | \ FM\subseteq M, \ VM\subseteq M,\  \Pi M\subseteq M,\  M=M^{\sharp},  \ [M+\tau M:M] \leq 1 \}.$$

\begin{corollary}
For $k$ and $\Lambda$ as above, the set  $\tilde\CN_{\Lambda} (k)$ is in natural bijection with the set of $W(k)$-lattices
$$\{ \ M\subseteq \Lambda_k   \ | \  M=M^{\vee},  \ [M+\tau M:M] \leq 1 \}.$$
\end{corollary}
\begin{proof}
This follows from Lemma \ref{dieud}.
\end{proof}

Let $V=\Lambda/\Lambda^{\sharp}=\Lambda/\Lambda^{\vee}$ (hence $V$ is no longer the {\it Verschiebung}!). It is a $\BF_p$-vector space of dimension $2d$. Then $V_k:=V\otimes k=\Lambda_k/\Lambda_k^{\vee}$. We define forms $\langle\, ,\, \rangle_V$ on $V$ resp. $V_k$ as follows: For $\bar x, \bar y\in V$ with lifts $x,y\in \Lambda$ let $\langle x, y\rangle_V $ be the image of $p\langle x,y\rangle$ in $\BF_p$. Extend this form $k$-linearly to $V_k$. Note that $\tau$ induces the identity on $V$, and the  Frobenius on $V_k$. 
\begin{lemma}
The form $\langle \, , \, \rangle_V $ is a non-degenerate alternating form.
\end{lemma}
\begin{proof}
It is clear that $\langle\ ,\ \rangle_V$ is alternating. To see that it is non-degenerate, let $\bar x\in V_k$  and suppose that $\langle \bar x, \bar y\rangle_V=0$ for any $\bar y\in V_k$. It follows that for any lift $x\in \Lambda_k$ of $\bar x$ and any $y\in \Lambda_k$ we have $p\langle x,y\rangle \in (p)$. Hence $x\in \Lambda_k^{\vee}= \Lambda_k^{\sharp}$, hence $\bar x=0$.
\end{proof}

\begin{lemma}\label{kwert}
The map $f: M\mapsto M/\Lambda^{\sharp}_k$ induces a bijection between the set $\CV_{\Lambda}(k)$ and the set  $S_V(k)$.
\end{lemma}
\begin{proof}

Let $M\in \CV_{\Lambda}(k)  $ and let $U=M/\Lambda^{\sharp}_k$, which we regard as a subspace of $V_k$. Then, since $M=M^{\vee}$ and by the definition of the symplectic form on $V_k$, we have $U=U^{\perp}$. The dimension of $U$ is the index of  $\Lambda^{\sharp}_k$ in $M$ which is $d$.
The condition $ [M+\tau M:M] \leq 1 $ carries over to $U$, so $U$ lies indeed in $S_V(k)$. 
Next we show that $f$ is a bijection: Suppose that $M_1/\Lambda^{\sharp}_k=M_2/\Lambda^{\sharp}_k$. Then $M_1=M_2$ since $\Lambda^{\sharp}_k \subseteq M_1\cap M_2$, hence $f$ is injective.
Next, let $U\in S_V(k)$. Let $M=q^{-1}(U)$, where $q: \Lambda_k \rightarrow V_k$   is the canonical projection. We claim that $M\in \CV_{\Lambda}(k)$.  The condition $ [U+\tau U:U] \leq 1 $ carries over to $M$. We have to show that $M=M^{\vee}$. If $x,y\in M$ with images $\bar x, \bar y\in U$,  then $\langle\bar x, \bar y \rangle_V=0$, hence $p\langle x,y\rangle \equiv 0 \mod p$, hence $\langle x,y \rangle $ is integral, so $M\subseteq M^{\vee}$. Conversely, suppose that $x\in M^{\vee}$,  and let $\bar x$ be its image in $V_k$. Let $\bar y$ be any element of $U$. Then $\langle\bar x , \bar y \rangle_V=0$. Since $\bar y$ was arbitrary, $\bar x \in U^{\perp}=U$, hence $x\in M$. Thus     $M^{\vee}\subseteq M$. We conclude that $M\in \CV_{\Lambda}(k)$ is a preimage of $U$.
\end{proof}

For an $\BF$-scheme $S$ and a $p$-divisible group $X$ over $S$ denote by $D(X)$ the Lie algebra of the universal vector extension of $X$. It is a locally free $\CO_S$-module of rank equal to the height of $X$ and it is compatible with base change.

We consider the following functor $\mathfrak{Grass}(V)$ on the category of $\BF$-algebras. For   an $\BF$-algebra $R$ let $\mathfrak{Grass}(V)(R)$ be the set of  locally direct summands $U$ of $V\otimes R$ of rank $d$.  Then $\mathfrak{Grass}(V)$ is represented by an $\BF$-scheme. 

Next we define a morphism $\tilde f:\widetilde{ \CN}_{\Lambda}\rightarrow \mathfrak{Grass}(V)$.
Let $R$ be an 
$\BF$-algebra and let $(X,\iota, \lambda, \rho_X)\in \widetilde{\CN}_{\Lambda}(R)$. By the same reasoning as in \S 4.6 of \cite{VW}, $E(X):= \ker(D(\rho_{\Lambda^-, X}))$ is a direct summand of $( \Lambda^+/ \Lambda^-)\otimes_{ \BF} R=V\otimes_{\BF_p} R$. Its rank is the height of $\rho_{\Lambda^-,X}$, which is $t(\Lambda)/2=d$, as remarked above.
 
Thus we obtain a map $\widetilde\CN_{\Lambda}(R)\rightarrow \mathfrak{Grass}(V)(R)$ by sending $(X,\iota,\lambda, \rho)$ to $E(X)$. This induces the desired map $\tilde f:\widetilde{ \CN}_{\Lambda}\rightarrow \mathfrak{Grass}(V)$.
Let now ${ \CN}_{\Lambda}$ be the underlying reduced subscheme of $\widetilde{ \CN}_{\Lambda}$. We have 
canonical inclusions of functors ${ \CN}_{\Lambda}\hookrightarrow \widetilde{ \CN}_{\Lambda}$ and $S_V\hookrightarrow \mathfrak{Grass}(V)$. 
\begin{lemma}
The restriction of $\tilde f$ to ${ \CN}_{\Lambda}$ induces a map $f:{ \CN}_{\Lambda}\rightarrow S_V$. 
\end{lemma}
\begin{proof}
We can test this on $\BF$-valued points. But then the claim follows from Lemma \ref{kwert}.
\end{proof} 
\begin{proposition}\label{iso}
The map $f: \CN_{\Lambda}\rightarrow S_V$ is an isomorphism. 
\end{proposition}

\begin{proof}
By Lemma \ref{kwert}, we know that $f$ is a bijection on $k$-valued points for any algebraically closed field $k$ of characteristic $p$. 
 Next, as in \cite{VW}, using Zink's theory of windows for displays \cite{Zi, Zi2}, one  checks that this reasoning carries over to $k$-valued points for any  field $k$ containing $\BF$. Since $S_V$ is normal (Proposition \ref{SDLV}) and $f$ is proper (as a morphism between projective varieties), it follows from Zariski's main theorem that $f$ is an isomorphism.
\end{proof}

\begin{corollary}\label{ProjectiveNormal}
The scheme $\CN_{\Lambda}$ is projective and normal.
\end{corollary}
\begin{proof}
Combine Proposition \ref{SDLV}, Lemma \ref{rep} and Proposition \ref{iso}.
\end{proof}

\begin{proposition}
The isomorphism $f$ is compatible with inclusions of vertex lattices. 
\end{proposition}
\begin{proof}
We can check this on $k$-valued points for $k$ algebraically closed, where it is obvious from the definition of the map $f$. 
\end{proof}

\noindent Let $$\CN^{\circ}_{\Lambda}=\CN_{\Lambda}\setminus \bigcup_{\Lambda'\varsubsetneq \Lambda}\CN_{\Lambda'}.$$
For any stratified space, we define the following simplicial complex, called its {\it incidence complex}: The simplices are formed by finite sets of strata such that the intersection of their closures is non-empty.

\noindent  Now we are ready to state the main theorem. It summarizes the results of this section and of the preceding sections.
\begin{theorem}\label{mainth}
i) There is a stratification of $\bar\CN^{0}_{\rm{red}}$ by locally closed subschemes 
given by $$
\bar\CN^{0}_{\rm{red}}= \biguplus_{\Lambda}\CN^{\circ}_{\Lambda}\, , 
$$
where $\Lambda$ runs over all vertex lattices. Each stratum $\CN^{\circ}_{\Lambda}$ is isomorphic to a Deligne-Lusztig variety attached to the symplectic group of size $t(\Lambda)$ over $\BF_p$ and a standard Coxeter element. The closure of any $\CN^{\circ}_{\Lambda}$ in $\bar\CN^{0}_{\rm{red}}$ is given by 
$$
\overline{\CN^{\circ}_{\Lambda}}=\biguplus_{ \Lambda'\subseteq \Lambda}\CN^{\circ}_{\Lambda'}=\CN_{\Lambda}\, ,
$$
and this is a projective variety of dimension   $t(\Lambda)/2$, which is normal 
with isolated singularities when $t(\Lambda)\geq 4$. When $t(\Lambda)=2$, then $\overline{\CN^{\circ}_{\Lambda}}\cong \BP^1$.

\smallskip

\noindent We call this stratification the \emph{Bruhat-Tits stratification} of $\bar\CN^{0}_{\rm{red}}$. 

\smallskip

\noindent ii) The incidence complex of the Bruhat-Tits stratification can be identified with the simplicial complex of vertex lattices $\CT$, cf. Proposition \ref{Gebaeude}.  In particular,  $\bar\CN^{0}_{\rm{red}}$ is connected.  When $n$ is odd, or $n$ is even and $C$ is non-split, then this can also be identified with the Bruhat-Tits simplicial complex of $\SU(C)(\BQ_p)$.

\smallskip

\noindent iii) The $\CN_{\Lambda}$, for $\Lambda$ of  maximal type,  are the irreducible components of $\bar\CN^{0}_{\rm{red}}$. The maximal value $t_{\max}$ of $t(\Lambda)$ is given by
$$t_{\max}=
\begin{cases}
   n  & \text{if $n$ is even and $C$ is split,}\\  
   n-2 & \text{if $n$ is even and  $C$ is non-split,}\\    
      n-1 & \text{if $n$ is odd.}\\
\end{cases}$$
 
 \noindent In particular,  $\bar\CN^{0}_{\rm{red}}$ is of pure dimension $t_{\max}/2$. 
  
\smallskip

\noindent iv) Let $\Lambda$ be a vertex lattice of type $2m$ and let $V=\Lambda/\Lambda^{\sharp}$. The scheme  $\overline{\CN^{\circ}_{\Lambda}}=\CN_{\Lambda}$ is isomorphic to the projective closure of a generalized Deligne-Lusztig variety $S_V$.  Furthermore the stratification by Deligne-Lusztig varieties $S_V= \biguplus_{i = 0}^m S_i$ of Proposition \ref{decomp} is compatible with the stratification given in i) in the following sense: Under the isomorphism  $\overline{\CN^{\circ}_{\Lambda}}\cong S_V$ (given by Proposition \ref{iso}), for any $i\leq m$ and any inclusion of index $m-i$ of a vertex lattice $\Lambda'\subseteq \Lambda$, the subscheme $\CN^{\circ}_{\Lambda'}$ is identified with an irreducible component of $S_i$, and all  irreducible components of $S_i$ arise in this way. In particular, $S_m$  is identified with $\CN^{\circ}_{\Lambda}$.

\begin{proof} 
The claims in {\emph{i)}}  follow from Proposition \ref{pointstrat}, the construction of $\CN^{\circ}_{\Lambda}$, Lemma \ref{rep}, and Propositions \ref{iso}, \ref{SDLV} and \ref{decomp}.
The claims in {\emph{ii)}}   follow from Proposition \ref{Gebaeude}.
The claim on the value of $t_{\max}$ is proved in Lemma \ref{type}. The claims on the dimension and on the irreducible components then  follow from 
Propositions \ref{dimS} and \ref{iso}. The connectedness of $\bar\CN^{0}_{\rm{red}}$ follows from the connectedness of the simplicial complex $\CT$, cf. Proposition \ref{Gebaeude}. 
The first claim in {\emph{iv)}} is given by Proposition \ref{iso}. 
Next we prove the claim on the compatibility in {\emph{iv)}}. For an inclusion of index $m-i$ of  vertex lattices $\Lambda'\subseteq \Lambda$ one gets an isotropic subspace $W$ of $V$, namely the orthogonal complement of the image of $\Lambda'$ under the projection  $\Lambda \rightarrow V$. The irreducible components of $S_i$ correspond to the isotropic subspaces of $V$ of dimension $m-i$, compare Proposition \ref{irred}. 
Conversely, given an  irreducible component of $S_i$ and hence an isotropic subspace $W$ of $V$, let $\Lambda '$ be the preimage of $W^{\vee}$ under the projection $\Lambda \rightarrow V$. Then one easily checks that $\Lambda '$ is indeed a vertex lattice. 
Given a point  in the irreducible component of $S_i$ associated to $W$ which  corresponds to a  flag $W=F_{m-i}\subset\ldots\subset F_m$,  one easily checks that the point in $\bar\CN^0$, with Dieudonn\'e Module 
$f^{-1}(F_m)$,  belongs to $\CN^{\circ}_{\Lambda '}$. (Here $f$ is the map of Lemma \ref{kwert}.) Similarly, any point in ${\CN^{\circ}_{\Lambda'}}$ is mapped to a point in the irreducible component of $S_i$ which corresponds to the space $W$ associated to $\Lambda'$ as above.
\end{proof}

\end{theorem}

\section{Applications to Shimura varieties}\label{section.Shimuravarieties}

In this section we apply the results of the previous sections to Shimura varieties of Picard type. 

Let us recall from \cite{KR2} the global moduli problem that describes such a Shimura variety. (A similar summary is made in \cite{T}.)
Let $\bk$ be an imaginary quadratic field and denote by $\CO_{\bk}$ its ring of integers. Let $a\mapsto \bar a$ be the non-trivial Galois automorphism of $\bk$. Let $n\geq 1$. We consider the Deligne-Mumford stack $\CM(1,n-1)$ over  $\CO_{\bk}$ whose functor of points associates to any  locally noetherian $\CO_{\bk}$-scheme $S$ the groupoid 
of tuples $(A,\iota, \lambda),$ where $A$ is an abelian scheme over $S$, where $\iota:  \CO_{\bk} \rightarrow \End_S(A)$ is an  $\CO_{\bk}$-action on $A$ satisfying the signature condition $(1,n-1)$, i.e.,  the characteristic polynomial of the action of $\iota(a)$ on the Lie algebra is of the form
$$
\text{charpol} (\iota(a), \Lie A)(T)=(T-a)(T-\bar a)^{n-1} \in \mathcal{O}_S [ T ], 
$$
 and where $\lambda: A \rightarrow A^{\vee}$ is a principal polarization for which the Rosati involution satisfies $\iota(a)^*=\iota(\bar a).$ Further, we require  that that the $\CO_{\bk}$-action on $\Lie  A$ satisfies  $\wedge^{n}(\iota(a)-a)=0$ and $\wedge^{2}(\iota(a)-\bar a)=0$ for $n\geq 3$. (Note that we have switched the signature $(n-1,1)$ used in part IV of \cite{KR2} to $(1,n-1)$. This conforms with the paper of Vollaard/Wedhorn \cite{VW}. On the other hand, if considering special cycles on these Shimura varieties in the framework of the Kudla program, it is more natural to consider the signature $(n-1, 1)$. Note that  both stacks are isomorphic via the isomorphism sending $(A,\iota, \lambda)$ to $(A,\bar \iota, \lambda)$, where $\bar\iota$ arises from $\iota$ by precomposing with the non-trivial Galois 
 automorphism of $\bk$.)

As in \cite{KR2}, we denote by  $\mathcal{R}_{(1,n-1)}(\bk)$ the set of {\it relevant hermitian spaces}, i.e., the set of isomorphism classes of hermitian spaces $V$ over $\bk$ of dimension $n$ and signature $(1,n-1)$ which contain a self-dual $\CO_{\bk}$-lattice. We denote by   $\mathcal{R}_{(1,n-1)}(\bk)^{\sharp}$ the set isomorphism classes of pairs $(V, [[L]])$, where $V \in \mathcal{R}_{(1,n-1)}(\bk)$ and $[[L]]$ is a ${\rm U}(V)$-{\it genus} of self-dual hermitian lattices in $V$ (see \cite{KR2}). 
Two hermitian spaces $V,(\ ,\ )$ and $V',(\ ,\ )'$ are {\it strictly similar}, if $V\cong V'$ with $(\ ,\ )=c(\ ,\ )'$ for some $c\in \BQ_+^{\times}$.
As shown in \cite{KR2}, Proposition 2.20 there is a natural decomposition
$$
 \CM(1,n-1)[\frac{1}{2}]=\coprod_{V^{\sharp} \in\mathcal{R}_{(1,n-1)}(\bk)^{\sharp}/{\text{str.sim.}}}\CM(1,n-1)[\frac{1}{2}]^{V^{\sharp}}.
$$ 
Here $[\frac{1}{2}]$ refers to the base change over $\Spec \CO_{\bk}$ with $\Spec \CO_{\bk}[\frac{1}{2}]$. The decomposition is according to the classes in $\mathcal{R}_{(1,n-1)}(\bk)^{\sharp}/{\text{str.sim.}}$ defined by the Tate modules of  points in $\CM(1,n-1)[\frac{1}{2}]$.

For  $V^{\sharp} = (V,L)$, one can identify $\CM(1,n-1)[\frac{1}{2}]^{V^{\sharp}}\times_{{\Spec \CO_{\bk}[\frac{1}{2}]}} \Spec \bk$ with a model of the Shimura variety $\underline{\rm Sh}^V_K$ for ${\rm{GU}}(V)$ with respect to the stabilizer $K$ of $L$ in ${\rm{GU}}(V)(\BA_f)$. This is the canonical model in the sense of Deligne when $n\geq 3$, in which case the reflex field can be identified with $\bk$. When $n=2$, it is the model obtained from the canonical model by extension of scalars from the reflex field $\BQ$ to $\bk$. See \cite{KR2}, section 4, in particular Proposition 4.4, for the precise statement.

Now let $p>2$ be a prime which is ramified in $\bk.$

 We introduce deeper level structures. Let $K^p$ a compact open  subgroup of $K\cap {\rm{GU}}(V)(\BA^p_f)$. We will always take $K^p$ sufficiently small, in order to avoid torsion phenomena.  We consider the following moduli problem $\CM_{K^p}$ over the localization $\CO_{\bk _{(p)}}$of $\CO_{\bk}$ at $p$: It associates to an $\CO_{\bk _{(p)}}$-scheme $S$ the set of isomorphism classes of tuples $(A,\iota, \lambda, \eta^p)$, where $A,\iota, \lambda$  over $S$ are as above,  and where $\eta^p$ is a  $K^p$-level structure $\eta^p : T^p(A)\rightarrow L\otimes \hat\BZ^p \mod K^p$,  as explained in \cite{KR2}, section 13,  and also in \cite{VW}. The moduli problem $\CM_{K^p}$ is represented by a smooth, quasi-projective scheme over $\CO_{\bk _{(p)}}$ (since we are assuming that $K^p$ is small enough). Its generic fiber is a model  of the Shimura variety $\underline{\rm Sh}^V_{K'}$ where  $K'=K^pK_p$, with $K_p$ the stabilizer of the  self-dual lattice $L\otimes_{\CO_{\bk}}\CO_{\bk_p}$ in $V\otimes_{\bk} \bk_p$.

 We denote by $\CM_{K^p}^{ss}$ the supersingular locus of $\CM_{K^p}\times_{\Spec \CO_{\bk_{(p)}}}\Spec\, \overline{\BF}_p$ and by $\widehat\CM_{K^p}^{ss}$ the  completion of  $\CM_{K^p}\times_{\Spec \CO_{\bk_{(p)}}}\Spec (W(\overline{\BF}_p)\otimes_{\BZ_p} \CO_{\bk_p})$ along its  supersingular locus. The application we have in mind is the description of 
$\CM_{K^p}^{ss}$.

Let now $E=\bk_p$ and, as in section \ref{section.Themodulispace}, use the notation $\BF=\overline{\BF}_p$ and $\breve{E}= W({\BF})_{\BQ} \otimes_{\BQ_p} E$. 
Let $\BX$ be the $p$-divisible group of a fixed element $(A^o, \iota^o, \lambda^o, \eta^o)\in {\CM_{K^p}^{ss}}({\BF})$. Using $\BX$, we  define $\CN$ as in section \ref{section.Themodulispace}, as well as $C$ and $J(\BQ_p)$. 
\begin{remark} The discriminants  of the hermitian vector spaces $C$ and $V_p$ differ by the factor $(-1)^{n-1}$. Indeed, 
let  $(E,\iota_0,\lambda_0)$ be a fixed object in $\CM(0,1)(\BF)$,  and consider the $\bk$-vector space  $V'=\Hom_{\CO_{\bk}}^o(E,A^o)$ with hermitian form given by $h'(x,y)=\iota_0^{-1}(\lambda_0^{-1}\circ y^{\vee}\circ \lambda^o \circ x)$. Then one checks  that $C$ is  isomorphic to the space $V_p'$ (compare the proof of \cite{KR2}, Lemma 2.10 and also  \cite{KR}, Lemma 3.9 for the unramified case). On the other hand, $V'$ is positive definite and $V_{\ell}=V'_{\ell}$ for all primes $\ell\neq p$ (cf. \cite{KR2}, Lemma 2.7 and Lemma 2.10).  Since the signature of $V_\BR$ is $(1,n-1)$, the claim follows from a comparison of the product formulas for $V$ and $V'$. 

We recall from section \ref {section.Vertexlattices} (proof of lemma \ref{type})  that the invariant of $C$ is relevant only when $n$ is even (in which case the invariants of $C$ and $V_p$ are different); it is irrelevant when $n$ is odd. 
\end{remark}

One now shows in the same way as in the proof of \cite{KR2}, Theorem 5.5 that there is an isomorphism of formal schemes over $\CO_{\breve{E}}$,
$$
I^V(\BQ)\setminus (\mathcal{N}^0\times { G^V}(\BA_f^p)^o/K^p) \cong \widehat \CM^{ss}_{K^p}.
$$ 
Here $I^V(\BQ)$ denotes the group of quasi-isogenies of  $A^o$ that respect $\iota^o$ and $\lambda^o$, and ${ G^V}(\BA_f^p)^o$ the subgroup of ${ G^V}(\BA_f^p)$ where the multiplier is a unit in $\hat\BZ^p$. Also, we recall that $\CN^0$ is the open and closed formal subscheme of $\CN$ where the quasi-isogeny has height zero. 
Taking the reduced loci of both sides we obtain  for the supersingular locus,  
\begin{equation}\label{supsingloc}
I^V(\BQ)\setminus (\bar\CN^0_{\text{red}}\times { G^V}(\BA_f^p)^o/K^p) \cong \CM^{ss}_{K^p}.
\end{equation}

Let us describe the LHS in more concrete terms. Denote by $g_1,\dots,g_m \in G^V(\BA^p_f)$ representatives of the finitely many double cosets in $I^V(\BQ)\backslash G^V(\BA^p_f)^o/K^p$,  and set
\[
\Gamma_j = I^V(\BQ) \cap g_jK^pg^{-1}_j
\]
for $j = 1,\dots,m$. Then the subgroup $\Gamma_j \subset J(\BQ_p)$ is discrete and cocompact modulo center. By the assumption made on $K^p$ above, $\Gamma_j$ is torsion free for all $j$.  Then the LHS of \eqref{supsingloc} can be written as 
\[
\coprod_{j=1}^m \Gamma_j\backslash \bar\CN^0_{\rm red} \cong I^V(\BQ)\backslash (\bar\CN^0_{\rm red} \times G^V(\BA^p_f)^o/K^p) .
\]
Consider the induced surjective morphism
\begin{equation}\label{Psi}
\Psi\colon \coprod_{j=1}^m \bar\CN^0_{{\rm red}} \to  \CM^{ss}_{K^p}.
\end{equation}
As explained in \cite{VW}, 6.4.,  our assumption on $K^p$  implies that $\Psi$ is a local isomorphism, and that the restriction of $\Psi$ to any closed quasi-compact subscheme of $\bar\CN^0_{\rm red}$ is finite. Hence we may deduce from Theorem \ref{mainth} the following structure theorem. The formulation uses notation of Theorem \ref{mainth}.
\begin{theorem}
The supersingular locus $\CM^{ss}_{K^p}$ is of pure dimension $t_{\rm max}/2$. It is stratified by Deligne-Lusztig varieties to symplectic groups over $\BF_p$ of size varying between $0$ and $t_{\rm max}$ and standard Coxeter elements. The incidence complex of the stratification can be identified with the complex $\coprod_{j=1}^m \Gamma_j\backslash \CT$. 
\end{theorem}

Finally we apply our results on the structure of $\bar\CN^{0}_{\rm{red}}$ to describe the sets of irreducible components and connected components of $ \CM^{ss}_{K^p}$. Let $\Lambda\subset C$ be a vertex lattice of maximal type $t_{\max}$ and let $K_{J,p}$ be its stabilizer in $J(\BQ_p)$. Since all such vertex lattices are in one orbit of $\SU(C)(\BQ_p)$, the choice of $\Lambda$ is harmless. 
\begin{proposition}

\noindent (i) There is a natural bijection
$$
\{\text{irreducible components of } \CM^{ss}_{K^p}\} \leftrightarrow I^V(\BQ)\setminus \big(J(\BQ_p)^o/K_{J,p}\times { G^V}(\BA_f^p)^o/K^p\big).
$$

\smallskip

\noindent  (ii) There is a natural bijection $$
\{\text{connected components of }  \CM^{ss}_{K^p}\} \leftrightarrow I^V(\BQ)\setminus  { G^V}(\BA_f^p)^o/K^p.
$$
\end{proposition}

\begin{proof}
These statements   are proved in the same way as the corresponding statements in \cite{VW}, \S 6. Note for (ii)  that, by Theorem \ref{mainth}, (ii), $\bar\CN^{0}_{\rm{red}}$ is connected. \end{proof}


\bigskip
\obeylines
Mathematisches Institut der Universit\"at Bonn  
Endenicher Allee 60 
53115 Bonn, Germany.
email: rapoport@math.uni-bonn.de

\bigskip
\obeylines
Institut f\"ur Experimentelle Mathematik 
Universit\"at Duisburg-Essen, Campus Essen
Ellernstra{\ss}e 29
45326 Essen, Germany
email: ulrich.terstiege@uni-due.de

\bigskip
\obeylines
Mathematisches Institut der Universit\"at Bonn  
Endenicher Allee 60 
53115 Bonn, Germany.
email: s6sewils@uni-bonn.de

\end{document}